\documentclass[11pt,usenames,dvipsnames]{article}

\usepackage[english]{babel}
\usepackage{amsmath,amsthm,amssymb,bm,enumitem,mathrsfs,mathtools,thmtools,titlesec,tikz-cd,url}
\usepackage[left=2.4cm,right=2.4cm,top=2.4cm,bottom=2.4cm,bindingoffset=0cm]{geometry}

\usepackage[T1]{fontenc}
\usepackage{xcolor}
\usepackage{hyperref}
\hypersetup{colorlinks=true,citecolor=Blue,linkcolor=Sepia,urlcolor=Blue}

\usepackage{quoting}
\quotingsetup{vskip=0pt}

\setlist[enumerate]{topsep=2pt,label=\textup{(\arabic*)},leftmargin=2em,labelsep=.5em}
\setlist{noitemsep}

\declaretheoremstyle[
  spaceabove=\topsep, spacebelow=2ex,
  headfont=\normalfont\bfseries,
  notefont=\mdseries, notebraces={(}{)},
  bodyfont=\normalfont\itshape,
  postheadspace=.5em,
  qed=\qedsymbol
]{mystyle}

\declaretheoremstyle[
  spaceabove=\topsep, spacebelow=6pt,
  headfont=\normalfont\bfseries,
  notefont=\mdseries, notebraces={(}{)},
  bodyfont=\normalfont,
  postheadspace=.5em,
  qed=\qedsymbol
]{mydefstyle}

\theoremstyle{mystyle}
\declaretheorem[numberlike=subsection]{proposition}
\declaretheorem[numberlike=subsection]{theorem}
\declaretheorem[numberlike=subsection]{corollary}
\declaretheorem[numberlike=subsection]{lemma}

\theoremstyle{mydefstyle}
\declaretheorem[numberlike=subsection]{definition}
\declaretheorem[numberlike=subsection]{remark}

\numberwithin{equation}{subsection}

\titleformat{\section}[block]
  {\filcenter\normalfont\large\bfseries}{\thesection.}{.5em}{}
\titleformat{\subsection}[runin]
  {\normalfont\bfseries}{\thesubsection.}{.5em}{}

\titlespacing*{\section}{0pt}{5ex plus .2ex minus 1ex}{3ex plus .2ex minus 1ex}
\titlespacing*{\subsection}{0pt}{2ex plus .2ex minus 1ex}{.5em}

\newcommand{\bbF}{\mathbb{F}}

\newcommand{\bbN}{\mathbb{N}}

\newcommand{\bbQ}{\mathbb{Q}}

\newcommand{\bbZ}{\mathbb{Z}}

\newcommand{\cA}{\mathscr{A}}

\newcommand{\cF}{\mathscr{F}}

\newcommand{\cH}{\mathscr{H}}

\newcommand{\cL}{\mathscr{L}}

\newcommand{\cO}{\mathscr{O}}
\newcommand{\cP}{\mathscr{P}}

\newcommand{\cS}{\mathscr{S}}
\newcommand{\cT}{\mathscr{T}}

\newcommand{\cX}{\mathscr{X}}

\newcommand{\fsl}{\mathfrak{sl}}
\newcommand{\fgl}{\mathfrak{gl}}

\newcommand{\sF}{\mathsf{F}}

\newcommand{\CH}{\mathrm{CH}}

\newcommand{\Image}{\mathrm{Im}}

\newcommand{\Ker}{\mathrm{Ker}}

\newcommand{\Pic}{\mathrm{Pic}}

\newcommand{\Spec}{\mathrm{Spec}}
\newcommand{\St}{\mathrm{St}}

\newcommand{\Sym}{\mathrm{Sym}}

\newcommand{\ch}{\mathrm{ch}}

\newcommand{\dR}{\mathrm{dR}}
\newcommand{\id}{\mathrm{id}}
\newcommand{\inv}{\mathrm{inv}}

\newcommand{\pr}{\mathrm{pr}}

\newcommand{\Td}{\mathrm{Td}}
\newcommand{\fs}{\mathfrak{s}}
\newcommand{\fTd}{\mathfrak{Td}}
\newcommand{\fCT}{\mathfrak{CT}}

\newcommand{\tto}{\longrightarrow}
\newcommand{\isomarrow}{\xrightarrow{\,\sim\, }}

\let\epsilon=\varepsilon
\let\phi=\varphi

\begin{document}

\begin{center}
\textbf{\LARGE Integral aspects of Fourier duality for abelian varieties}
\bigskip

\textit{by}
\bigskip

{\Large Junaid Hasan, Hazem Hassan\footnote{Partially supported by a Graduate Student Scholarship of the Institut des Sciences Math\'ematiques au Qu\'ebec.}, Milton Lin,}
\medskip

{\Large Marcella Manivel\footnote{Supported by the NSF Graduate Research Fellowship Program, Award \# 2237827}, Lily McBeath\footnote{Supported by NSF grant DMS-2200845}, and Ben Moonen}
\end{center}
\vspace{8mm}

{\small 

\noindent
\begin{quoting}
\textbf{Abstract.} We prove several results about integral versions of Fourier duality for abelian schemes, making use of Pappas's work on integral Grothendieck--Riemann--Roch. If $S$ is smooth quasi-projective of dimension~$d$ over a field and $\pi \colon X\to S$ is a $g$-dimensional abelian scheme, we prove, under very mild assumptions on~$X/S$, that all classical results about Fourier duality, including the existence of a  Beauville decomposition, are valid for the Chow ring $\CH(X;\Lambda)$ with coefficients in the ring $\Lambda = \bbZ[1/(2g+d+1)!]$. If $X$ admits a polarization~$\theta$ of degree~$\nu(\theta)^2$ we further construct an $\fsl_2$-action on $\CH(X;\Lambda_\theta)$ with $\Lambda_\theta = \Lambda[1/\nu(\theta)]$, and we show that $\CH(X;\Lambda_\theta)$ is a sum of copies of the symmetric powers $\Sym^n(\St)$ of the $2$-dimensional standard representation, for $n=0,\ldots,g$. For an abelian variety over an algebraically closed field, we use our results to produce torsion classes in~$\CH^i(X;\Lambda_\theta)$ for every $i\in \{1,\ldots,g\}$.
\medskip

\end{quoting}

} 
\vspace{4mm}

\section{Introduction}

\subsection{} A powerful tool in the study of algebraic cycles on abelian varieties is the Fourier transform, which first appeared in Mukai's seminal paper~\cite{Mukai}, and which on the level of Chow groups was studied by Beauville in his influential papers \cite{BeauvFourier} and~\cite{BeauvChow}. Arguably the main drawback of the Fourier transform on the Chow ring is that it works with rational coefficients, so that we lose all information about torsion classes.

Already in Beauville's paper~\cite{BeauvFourier}, it is observed that for an abelian variety~$X$ with dual~$X^t$, one can define homomorphisms $\sF \colon \CH(X) \to \CH(X^t)$ and $\sF^t \colon \CH(X^t) \to \CH(X)$ on integral Chow rings that lift $N$ times the classical Fourier transforms, for some positive integer~$N$, and that have all the expected properties. However, the value of~$N$ that one can take (which will depend on $g=\dim(X)$) is not made effective. If one attempts to make this effective, there are two points that need to be considered. The first is that the classical Fourier transform $\cF\colon \CH(X;\bbQ) \to \CH(X^t;\bbQ)$ is defined as the homomorphism that is induced by the correspondence $\ch(\cP) \in \CH(X\times X^t;\bbQ)$, where $\cP$ is the Poincar\'e bundle on $X\times X^t$. To get rid of the denominators that appear, it suffices to multiply~$\ch(\cP)$ by a factor of~$(2g)!$. The second point, which is harder to control, is that the proof of the basic properties of~$\cF$, such as the fact that $\cF^t \circ \cF = (-1)^g \cdot [-1]^*$, relies in an essential way on the Grothendieck--Riemann--Roch theorem, which involves denominators.

The present paper, which finds its origin in a project at the Arizona Winter School 2024, is based on the idea that one can use the results of Pappas~\cite{Pappas} to obtain, for a given abelian variety~$X$, an effective number~$N>0$ such that $N$~times the classical Fourier transform can be lifted to an integral transformation with all the expected properties. We carry this out in the more general setting of abelian schemes $\pi\colon X\to S$ over a base scheme~$S$ that is smooth and quasi-projective over a field~$L$. As the results of Pappas are available only in characteristic~$0$, we have to make some assumptions on~$X/S$, but by using liftings to characteristic~$0$, we are able to prove our main results under some very mild assumptions (see Section~\ref{subsec:charL}).

Let us point out here that for a $g$-dimensional abelian scheme $\pi\colon X\to S$, the duality relation that, in some sense, comes first, is that for a class $\alpha \in \CH(X;\bbQ)$ one has
\begin{equation}\label{eq:FtFTodd}
(\cF^t\circ \cF)\bigl(\alpha\bigr) \cdot \pi^*\bigl(\Td(\cH)\bigr) = (-1)^g \cdot [-1]^*(\alpha)\, ,
\end{equation}
where $\cH = R^1\pi_*(\Omega^\bullet_{X/S})$ is the de Rham bundle, which is a vector bundle of rank~$2g$ on~$S$. The more familiar relation $(\cF^t\circ \cF)\bigl(\alpha\bigr) = (-1)^g \cdot [-1]^*(\alpha)$ was proven by Deninger and Murre in~\cite{DenMurre} by using more elaborate arguments, and it then follows from their results that in fact $\Td(\cH) =1$. (See~\cite{BM-RPDM} for further discussion.)

\subsection{}
The first main result we prove in the paper is Theorem~\ref{thm:FtF} in the text. We consider an abelian scheme $\pi\colon X\to S$, and we let $d=\dim(S)$ and $g = \dim(X/S)$. As mentioned, we need to impose some mild technical assumptions; e.g., it suffices if $X$ admits a separable polarization and a level~$n$ structure with $n\geq 3$ and $\mathrm{char}(L) \nmid n$. We then construct integral transformations $\sF \colon \CH(X) \to \CH(X^t)$ and $\sF^t \colon \CH(X^t) \to \CH(X)$ that lift $(2g+d)!$ times the classical Fourier transforms, such that an integral version of the relation~\eqref{eq:FtFTodd} holds.

Our second main theorem says that essentially all classical results due to Beauville remain valid for Chow groups with coefficients in the ring $\Lambda = \bbZ[\frac{1}{(2g+d+1)!}]$. Writing $\CH(X;\Lambda) = \CH(X) \otimes \Lambda$, we define transformations
\[
\cF \colon \CH(X;\Lambda) \to \CH(X^t;\Lambda)\, ,\qquad \cF^t \colon \CH(X^t;\Lambda) \to \CH(X;\Lambda)
\]
that lift the classical Fourier transforms, and we show that $\cF^t\circ \cF = (-1)^g \cdot [-1]^*$ and that $\Td(\cH) = 1$ in~$\CH(S;\Lambda)$. Moreover, if we define
\[
\CH^i_{(s)}(X;\Lambda) = \bigl\{x\in \CH^i(X;\Lambda) \bigm| [n]_X^*(x) = n^{2i-s} \cdot x\ \text{for all $n\in \bbZ$} \bigr\} \subset \CH^i(X;\Lambda)\, ,
\]
we prove that there is a Beauville decomposition
\[
\CH^i(X;\Lambda) = \bigoplus_{s=\max\{i-g,2i-2g\}}^{\min\{i+d,2i\}}\; \CH^i_{(s)}(X;\Lambda)
\]
with all the usual properties. This decomposition is compatible with the intersection product and the Pontryagin product, and if $Y/S$ is another abelian scheme of dimension~$\leq g$, we have the expected compatibility of all constructions with $f^*$ and~$f_*$. We refer to Theorem~\ref{thm:MainLambda} for the precise statements.

\subsection{}
Going even further, we also construct an `integral' version of the $\fsl_2$-action on the Chow ring. In addition to the conditions that were already mentioned, we now assume that $X$ has a polarization $\theta \colon X\to X^t$, and we work with coefficients in the ring $\Lambda_\theta = \Lambda[\nu(\theta)^{-1}]$, where $\nu(\theta) = \deg(\theta)^{1/2}$. {}From~$\theta$ we construct classes $\ell \in \CH^1_{(0)}(X;\Lambda_\theta)$ and $\lambda \in \CH^{g-1}_{(0)}(X;\Lambda_\theta)$, and we then define endomorphisms $e$, $h$ and~$f$ of~$\CH(X;\Lambda_\theta)$ by
\[
e(x) = \ell\cdot x\, ,\qquad h(x) = (2i-s-g)\cdot x\quad \text{for $x \in \CH^i_{(s)}(X;\Lambda_\theta)$}\, ,\qquad f(x) = \lambda \star x\, .
\]
In Theorem~\ref{thm:sl2} we prove that these endomorphisms define a representation of the Lie algebra~$\fsl_2$ over the ring~$\Lambda_\theta$ on $\CH(X;\Lambda_\theta)$. All subspaces $\CH^\bullet_{(s)}(X;\Lambda_\theta) = \oplus_i\; \CH^i_{(s)}(X;\Lambda_\theta)$ are $\fsl_2$-submodules. Moreover, as representations of~$\fsl_2$, these are all `sums' of copies of the symmetric powers~$\Sym^n(\St)$, for $n=0,\ldots,g$, of the $2$-dimensional standard representation. More precisely, if for $n=0,\ldots,g$ we define
\[
M_n = \bigl\{x \in \oplus_{2i-s=g-n}\; \CH^i_{(s)}(X;\Lambda_\theta) \bigm| f(x) = 0\bigr\}
\]
then we have
\[
\CH(X;\Lambda_\theta) \cong \bigl(\Sym^0(\St) \otimes_{\Lambda_\theta} M_0\bigr) \oplus \bigl(\Sym^1(\St) \otimes_{\Lambda_\theta} M_1\bigr) \oplus \cdots \oplus \bigl(\Sym^g(\St) \otimes_{\Lambda_\theta} M_g\bigr)
\]
as representations of~$\fsl_2$. Note that in general the $\Lambda_\theta$-modules~$M_n$ are neither of finite type nor torsion-free.

\subsection{}
In the final section, we consider an abelian variety~$X$ over an algebraically closed field~$L$. As an application of our results, we show that $\CH^i_{(s)}(X;\Lambda_\theta)$ contains a copy of $X(L) \otimes \Lambda_\theta$ for every $i \in \{1,\ldots,g\}$. In particular, if $p>2g+1$ is a prime number different from~$\mathrm{char}(L)$ with $p\nmid \nu(\theta)$ then $\CH^i_{(s)}(X;\Lambda_\theta)$ contains a subgroup isomorphic to~$(\bbQ_p/\bbZ_p)^{2g}$.

\subsection{}
To conclude this introduction, let us mention that there are other approaches to an integral version of Fourier duality for abelian varieties. For instance, as shown by Polishchuk and one of us in~\cite{DPCRIFT}, the Pontryagin product on~$\CH(X)$ admits divided powers, and the Fourier transform can be defined using only the Pontryagin product. Some very interesting applications of this have been given by Beckmann and de Gaay Fortman in~\cite{BeckdGF}. It could be of interest to see if this perspective can be combined with the methods of the present paper, but we have not yet pursued this.

\subsection{Acknowledgements.} This paper is the outgrowth of a project in the Arizona Winter School 2024 that was led by one of us (B.M.) and in which the other authors participated. We should like to thank the organizers of the AWS, as well as the other lecturers and all participants, for the wonderful and inspiring event. JH, HH,  ML, MM and LM would like to sincerely thank BM for bringing us onto the project at the AWS, as well as for continuing to share his time, energy, and support.

\subsection{Notation and conventions.} If $X$ is a smooth variety over a field and $\Lambda$ is a commutative ring, we define $\CH(X;\Lambda) = \CH(X) \otimes \Lambda$, where $\CH(X)$ is the group of algebraic cycles on~$X$ modulo rational equivalence.

If $X$ is an abelian scheme over some base scheme and $n\in \bbZ$, we denote by~$[n]_X$, or simply~$[n]$, the endomorphism of~$X$ given by multiplication by~$n$.

\section{Review of some results of Pappas}

In this section, we review some results of Pappas~\cite{Pappas} about an integral version of the Grothendieck--Riemann--Roch theorem.

\subsection{}\label{subsec:Tm}
For $m\in \bbN$, define
\[
T_m = \prod_p\; p^{\lfloor \frac{m}{p-1}\rfloor}\; .
\]
By \cite{Pappas}, Lemma~2.1, if $m_1,\ldots,m_r, n_1,\ldots,n_s$ are natural numbers with $m_1 + \cdots + m_r + n_1 + \cdots + n_s \leq m$ then $(m_1+1)! \cdots (m_r+1)! \cdot T_{n_1} \cdots T_{n_s}$ divides~$T_m$.

\begin{lemma}\label{lem:Nh!2}
Let $h$ be a positive integer. Let
\[
N = \begin{cases}
2 & \text{if $h=3$,}\\
h+1 & \text{if $h+1$ is prime,}\\
1 & \text{otherwise.}
\end{cases}
\]
Then $T_h$ divides $N\cdot h!^2$.
\end{lemma}

\begin{proof}
We need to show that $v_p(N\cdot h!^2) \geq v_p(T_h)$ for all primes~$p$. By definition, $v_p(T_m) = \lfloor \frac{m}{p-1}\rfloor$, whereas $v_p(m!) = \sum_{i\geq 1}\; \lfloor \frac{m}{p^i}\rfloor$. If $p> h+1$ then $v_p(T_h) = 0$. If $p = h+1$ is prime then $v_p(N\cdot h!^2) = 1 = v_p(T_h)$. Next suppose $p\leq h$. If $h = a_0 + a_1 p + \cdots a_t p^t$ with $a_0, \ldots, a_t \in \{0,\ldots,p-1\}$ then
\[
v_p(T_h) = \left\lfloor\tfrac{a_0+\cdots + a_t}{p-1}\right\rfloor + \sum_{j=1}^t\; a_j \cdot \tfrac{p^j-1}{p-1}\, ,\quad\text{and}\quad v_p\bigl(h!\bigr) = \sum_{j=1}^t\; a_j \cdot \tfrac{p^j-1}{p-1}\; .
\]
The case $p=2$ is easily checked by hand. If $p>2$ (so that $v_p(N) = 0$) then $v_p(T_h) > v_p(N\cdot h!^2)$ if and only if
\begin{equation}\label{eq:ineq}
\left\lfloor\frac{a_0+\cdots + a_t}{p-1}\right\rfloor > \sum_{j=1}^t\; a_j \cdot \left(\frac{p^j-1}{p-1}\right)\; .
\end{equation}
If this holds, it also holds with $a_0$ replaced by $p-1$; but in that case \eqref{eq:ineq} is equivalent to $a_1+\cdots + a_t \geq \sum_{j=1}^t\; a_j \cdot (p^j-1)$. Because $p>2$, this happens only if $a_1 = \cdots = a_t = 0$, which contradicts the assumption that $p\leq h$.
\end{proof}

\subsection{}\label{subsec:Zpolynoms}
We view the Todd class of a vector bundle as a polynomial in its Chern classes:
\[
\Td = 1 + \tfrac{1}{2} \cdot c_1 + \tfrac{1}{12}\cdot (c_1^2+c_2) + \tfrac{1}{24}\cdot (c_1c_2) + \tfrac{1}{720}\cdot (-c_1^4 + 4c_1^2c_2 + 3c_2^2 +c_1c_3 - c_4) + \cdots \; .
\] 
The term~$\Td_m$ in degree~$m$ is a polynomial in $\bbQ[c_1,c_2,\ldots,c_m]$ whose denominator equals~$T_m$. Define $\fTd_m \in \bbZ[c_1,c_2,\ldots,c_m]$ by
\[
\fTd_m = T_m \cdot \Td_m\, .
\]

Similarly, the Chern character
\[
\ch = r + c_1 + \tfrac{1}{2}\cdot (c_1^2-2c_2) + \tfrac{1}{6}\cdot (c_1^3 -3c_1c_2 + 3c_3) + \cdots
\] 
of a vector bundle is a polynomial with rational coefficients in the rank~$r$ and the Chern classes. The denominator of the degree~$m$ term~$\ch_m$ is~$m!$. Let $r,c_1^\prime, c_2^\prime,\ldots,c_m^\prime$ be a new set of variables and define $\fs_m \in \bbZ[r,c_1^\prime,\ldots,c_m^\prime]$ by
\[
\fs_m(r,c_1^\prime,\ldots,c_m^\prime) = m! \cdot \ch_m(r,c_1^\prime,\ldots,c_m^\prime)\, .
\]

Next, define
\[
\fCT_m = \sum_{j=0}^m\; \frac{T_m}{j! \cdot T_{m-j}} \cdot (\fs_j \cdot \fTd_{m-j}) \, ,
\]
which is a polynomial in $\bbZ[c_1,c_2,\ldots,r,c_1^\prime,c_2^\prime,\ldots]$. (By the remark in~\ref{subsec:Tm}, all $\tfrac{T_m}{j! \cdot T_{m-j}}$ are integers.)

\subsection{}
Let $Y$ be a variety over a field~$L$ with tangent bundle~$\cT_Y$. If $\cF$ is a vector bundle of rank~$r$ on~$Y$, define elements $\fCT_m(Y,\cF)$ and $\fs_m(Y,\cF)$ of~$\CH(Y)$ by
\[
\fCT_m(Y,\cF) = \fCT_m\bigl(c_1(\cT_Y),\ldots,c_m(\cT_Y),r,c_1(\cF),\ldots,c_m(\cF)\bigr)
\]
and
\[
\fs_m(Y,\cF) = \fs_m\bigl(r,c_1(\cF),\ldots,c_m(\cF)\bigr)\, .
\]
(In particular, if $\cF$ is a line bundle then $\fs_m(Y,\cF) = c_1(\cF)^m$.) These elements only depend on the class of~$\cF$ in the Grothendieck group~$K_0(Y)$. If $Y/L$ is smooth, $K_0(Y)$ is isomorphic to the Grothendieck group of coherent $\cO_Y$-modules; therefore, the definitions of $\fCT_m(Y,\cF)$ and $\fs_m(Y,\cF)$ make sense for coherent sheaves~$\cF$. In this situation, we simply write~$K(Y)$ for either Grothendieck group, and we write $[\cF]$ for the class of~$\cF$. 

If $f \colon X\to S$ is a projective morphism of smooth varieties over~$L$ and $\cF$ is a coherent $\cO_X$-module, we define $f_*[\cF] = \sum_{i\geq 0}\; (-1)^i [R^if_*\cF]$. Further, we define $\fCT_m(X/S,\cF)$ by evaluating $\fCT_m$ at $c_i = c_i\bigl([\cT_X]-[f^*\cT_S]\bigr)$ and $c_j^\prime = c_j(\cF)$. If $f$ is smooth, $[\cT_X]-[f^*\cT_S] = [\cT_{X/S}]$.

\subsection{}
Let $Q(x) = \frac{x}{1-e^{-x}}$, thought of as a formal power series in~$x$, and let~$e_i$ be the $i$th elementary symmetric function. By definition, if $\alpha_1,\ldots,\alpha_r$ are the Chern roots of a vector bundle~$E$ then $c_i(E) = e_i(\alpha_1,\ldots,\alpha_r)$ and $\Td(E) = \prod_{i=1}^r\; Q(\alpha_i)$.

As explained in~\cite{Pappas}, Section~4.c, if $E$ is a vector bundle of rank~$r$ on a variety~$Y$ over~$L$ then
\[
\fTd^\inv_n(Y,E) = (n+r)! \cdot \Bigl\{\prod_{i=1}^r\; Q(\alpha_i)^{-1} \Bigr\}_n
\]
(where $\{~\}_n$ means the degree~$n$ part of $\prod\, Q(\alpha_i)^{-1}$ as a power series in the~$\alpha_i$) is a symmetric homogeneous polynomial in the~$\alpha_i$ with integral coefficients, and therefore defines an integral homogeneous polynomial expression in the Chern classes~$c_j(E)$. The `$\inv$' is for inverse; see~\ref{subsec:ids}\ref{TdTdinv}.

\subsection{Some identities.}\label{subsec:ids} We list some identities that we need later. To simplify expressions, we omit~$Y$ from the notation. Note that all coefficients that occur in these relations are integers.
\begin{enumerate}
\item\label{TdH} If $0 \tto E_1 \tto \cH \tto E_2 \tto 0$ is a short exact sequence,
\[
\fTd_i(\cH) = \sum_{j+k=i}\; \frac{T_i}{T_j\cdot T_k}\cdot \fTd_j(E_1) \cdot \fTd_k(E_2)\, .
\]

\item\label{TdETdEdual} For $E$ a vector bundle with dual bundle~$E^\vee$, we have
\[
\fTd_i(E) = \sum_{j+k=i}\; \frac{T_i}{T_j\cdot k!}\cdot \fTd_j(E^\vee) \cdot \fs_k\bigl(\det(E)\bigr)\, .
\]

\item\label{TdTdinv} If $E$ is a vector bundle of rank~$r$ and $m \in \bbN$ then
\[
\sum_{i=0}^m\; \frac{T_{r+m}}{T_i\cdot (r+m-i)!}\cdot \fTd_i(E) \cdot \fTd^\inv_{m-i}(E) = \begin{cases} T_{r+m} & \text{for $m=0$,} \\ 0 & \text{for $m\geq 1$.} \end{cases}
\]

\item\label{binom} If $\cL$ is a line bundle and $\delta = c_1(\cL)$ then we have the (obvious) identity
\[
\sum_{i=0}^m\; \binom{m}{i} \cdot \delta^i \cdot (-\delta)^{m-i} = \begin{cases} 1 & \text{for $m=0$,} \\ 0 & \text{for $m\geq 1$.} \end{cases}
\]
\end{enumerate}

\begin{theorem}[Pappas \cite{Pappas}, Corollary~2.3 and Proposition~4.4.]\label{thm:Pappas}
Let $X$ and~$S$ be smooth and quasi-projective varieties over a field~$L$ of characteristic~$0$.
\begin{enumerate}
\item\label{Pappas1} Let $f \colon X \to S$ be a smooth projective morphism of relative dimension~$g$. Let $\cF$ be a coherent $\cO_X$-module. Then in $\CH^n(S)$ we have the identity
\[
f_*\bigl(\fCT_{g+n}(X/S,\cF)\bigr) = \frac{T_{g+n}}{n!} \cdot \fs_n\bigl(S,f_*[\cF]\bigr)\, .
\]

\item\label{Pappas2} Let $\epsilon \colon Z\hookrightarrow X$ be a closed immersion of codimension~$g$ with normal sheaf~$N$, and let $\cF$ be a coherent $\cO_Z$-module. Then for $n\in \bbN$ with $n\geq g$ we have
\[
\fs_n\bigl(X,\epsilon_*[\cF]\bigr) =  \sum_{a=0}^{n-g}\; \binom{n}{a}\cdot \epsilon_*\bigl(\fs_a(Z,\cF) \cdot \fTd^\inv_{n-g-a}(Z,N) \bigr)\, ,
\]
and $\fs_n(X,\epsilon_*[\cF]) = 0$ if $n<g$. 
\end{enumerate}
\end{theorem}

\section{Application to abelian schemes}

We now apply the results of Pappas to abelian schemes. Proposition~\ref{prop:keyformula} is the main preparation for the results in the next sections. Throughout, $S$ denotes a connected quasi-projective variety which is smooth over a field~$L$. Let $d = \dim(S)$.

\subsection{}\label{subsec:setup}
Let $\pi\colon X \to S$ be an abelian scheme of relative dimension~$g>0$ with zero section $e\colon S\to X$. Let
\[
E_{X/S} = e^*\Omega^1_{X/S}\, ,
\]
which is a rank~$g$ vector bundle on~$S$, called the Hodge bundle of~$X$. Note that $E_{X/S} \cong \pi_*\Omega^1_{X/S}$. We have $\cT_{X/S} \cong \pi^*E_{X/S}^\vee$, so for $\cL$ a line bundle on~$X$ this gives
\begin{equation}\label{eq:fCTnX/S}
\fCT_n(X/S,\cL) = \sum_{i+j=n}\; \frac{T_n}{i^!\cdot T_j}\cdot c_1(\cL)^i \cdot \pi^*\fTd_j(S,E_{X/S}^\vee)\, .
\end{equation}

Let $\pi^t \colon X^t \to S$ be the dual abelian scheme. Its Hodge bundle~$E_{X^t/S}$ satisfies the relation $\det(E_{X/S}) \cong \det(E_{X^t/S})$; see \cite{GW}, Corollary~27.230. Let $\cP$ be the Poincar\'e line bundle on $X \times_S X^t$, and write $\wp = c_1(\cP) \in \CH^1(X\times_S X^t)$. For the pushforward of~$\cP$ under the first projection $\pr_1 \colon X\times_S X^t \to X$ we have (ibid., Theorem~27.243),
\begin{equation}\label{eq:Ripr1*P}
R^i\pr_{1,*}(\cP) = \begin{cases} 
0 & \text{for $i\neq g$}\\ 
e_*\bigl(\det(E_{X/S})^{-1}\bigr) & \text{for $i=g$.}
\end{cases}
\end{equation}

The de Rham bundle of $X$ over~$S$ is defined as
\[
\cH^1_\dR(X/S) = R^1\pi_*(\Omega^\bullet_{X/S})\, ,
\]
which is a vector bundle of rank~$2g$ over~$S$. By the degeneration of the Hodge--de Rham spectral sequence, it sits in a short exact sequence
\begin{equation}\label{eq:sesdR}
0 \tto E_{X/S} \tto \cH^1_\dR(X/S) \tto E_{X^t/S}^\vee \tto 0\, .
\end{equation}

\subsection{}\label{subsec:charL}
We now assume that we are in one of the following situations:
\begin{enumerate}[label=\textup{(\Alph*)}]
\item\label{charL=0} $\mathrm{char}(L) = 0$,

\item\label{lift} $\mathrm{char}(L) = p > 0$ and $X/S$ can be lifted to characteristic~$0$, by which we mean that there exists a discrete valuation ring~$R$ of mixed characteristic $(0,p)$ with residue field~$L$, a smooth $R$-scheme~$\cS$ and an abelian scheme $\cX\to \cS$, whose special fibre is (isomorphic to) the abelian scheme $X\to S$,

\item\label{basechange} $X\to S$ can be obtained by base-change from a case where \ref{lift} holds.
\end{enumerate}
In particular, this includes the case where $X$ is the universal abelian scheme over the moduli scheme $S = \cA_{g,\delta,n}$ of $g$-dimensional abelian varieties with a polarization of degree~$\delta^2$ and a level~$n$ structure, over $L=\bbQ$ or $L = \bbF_p$ with $p\nmid n\delta$. By using~\ref{basechange}, it follows that we are in the above situation as soon as~$X$ admits a polarization~$\theta$ and a level~$n$ structure with $n\geq 3$ and $\mathrm{char}(L) \nmid n\cdot \deg(\theta)$. At the other extreme, if $S = \Spec(L)$ is a point, then by the results of Norman~\cite{Norman}, either~\ref{charL=0} or~\ref{lift} is satisfied with the possible exception of the case when $\mathrm{char}(L) = 2$ and $X$ admits no polarization of degree prime to~$2$. Even in this last case, the results of Norman and Oort in~\cite{NormanOort} guarantee that $X$ can be lifted to characteristic~$0$ after replacing~$L$ by a finite extension, although allowing a field extension goes somewhat against the grain of studying integral Chow groups, as for a field extension $L \subset L^\prime$ the map $\CH(X) \to \CH(X_{L^\prime})$ is in general neither injective nor surjective. 
\medskip

The following result is the key ingredient for our results on Fourier duality in the next section. To simplify notation, we write $\fTd_j(\cH)$ instead of $\fTd_j(S,\cH)$, and likewise for the Hodge bundles~$E_{X/S}$ and~$E_{X^t/S}$.

\begin{proposition}\label{prop:keyformula}
Let notation be as in~\ref{subsec:setup}, and assume that one of the conditions in~\ref{subsec:charL} is satisfied. Let $\cH = \cH^1_\dR(X/S)$. For $m \in \bbN$ we have
\begin{equation}\label{eq:keyform}
\sum_{i+j=m}\; \frac{T_{g+m}}{(g+i)!\cdot T_j} \cdot \pr_{1,*}(\wp^{g+i}) \cdot \pi^*\fTd_j(\cH) = 
\begin{cases} 0 & \text{if $m\neq g$,} \\ (-1)^g \cdot T_{2g} \cdot \bigl[e(S)\bigr] & \text{if $m=g$,} \end{cases}
\end{equation}
in $\CH(X)$.
\end{proposition}

\begin{proof}
It suffices to prove this in case $\mathrm{char}(L) = 0$; in case~\ref{subsec:charL}\ref{lift}, the result then follows by specialization, and then it is clear that the identity also holds in case~\ref{basechange}.

Let $\delta = c_1\bigl(\det(E_{X/S})\bigr) \in \CH^1(S)$. By \ref{subsec:ids}\ref{TdH} and \ref{subsec:ids}\ref{TdETdEdual}, we can expand the LHS of~\eqref{eq:keyform} as
\begin{equation}\label{eq:expanded}
\sum_{i+j+k+\ell=m}\; \frac{T_{g+m}}{(g+i)! \cdot T_j \cdot T_k\cdot \ell!} \cdot \pr_{1,*}(\wp^{g+i}) \cdot \pi^*\Bigl(\fTd_j(E^\vee_{X^t/S}) \cdot \fTd_k(E^\vee_{X/S}) \cdot \delta^\ell \Bigr)\, .
\end{equation}

We apply Theorem~\ref{thm:Pappas}\ref{Pappas1} to $X\times_S X^t$, viewed as an abelian scheme over~$X$ via~$\pr_1$. Note that $\pr_{1,*}(\wp^q) = 0$ if $q<g$, for dimension reasons. Using~\eqref{eq:fCTnX/S}, we find
\[
\sum_{i+j=n}\; \frac{T_{g+n}}{(g+i)^!\cdot T_j}\cdot \pr_{1,*}(\wp^{g+i}) \cdot \pi^*\fTd_j(E_{X^t/S}^\vee) = (-1)^g \cdot \frac{T_{g+n}}{n!}\cdot \fs_n\bigl(X,e_*[\det(E_{X/S})^{-1}]\bigr)\, .
\]
Next we apply Theorem~\ref{thm:Pappas}\ref{Pappas2} to $e\colon S\to X$, which has normal bundle~$E^\vee_{X/S}$. For $n\geq g$ this gives
\begin{align*}
\fs_n\bigl(X,e_*[\det(E_{X/S})^{-1}]\bigr) &= \sum_{a=0}^{n-g}\; \binom{n}{a} \cdot e_*\bigl((-\delta)^a \cdot \fTd^\inv_{n-g-a}(E^\vee_{X/S})\bigr) \\
&= \sum_{a=0}^{n-g}\; \binom{n}{a} \cdot \pi^*(-\delta)^a \cdot \pi^*\bigl(\fTd^\inv_{n-g-a}(E^\vee_{X/S})\bigr) \cdot \bigl[e(S)\bigr]\, ,
\end{align*}
and $\fs_n\bigl(X,e_*[\det(E_{X/S})^{-1}]\bigr) = 0$ if $n<g$. Together, this gives for $n\geq g$ the identity
\begin{multline}\label{i+j=n}
\sum_{i+j=n}\; \frac{T_{g+n}}{(g+i)^!\cdot T_j}\cdot \pr_{1,*}(\wp^{g+i}) \cdot \pi^*\fTd_j(E_{X^t/S}^\vee) =\\
(-1)^g \cdot \frac{T_{g+n}}{n!}\cdot \sum_{a=0}^{n-g}\; \binom{n}{a} \cdot \pi^*(-\delta)^a \cdot \pi^*\bigl(\fTd^\inv_{n-g-a}(E^\vee_{X/S})\bigr) \cdot \bigl[e(S)\bigr]\, .\qquad
\end{multline}
Further, the LHS is zero if $n< g$. Using this in~\eqref{eq:expanded}, we already see that the LHS of~\eqref{eq:keyform} is zero if~$m<g$.

Next assume $m\geq g$, and write $m = g+\mu$. In~\eqref{eq:expanded}, we collect all terms for which $i+j$ is constant, say $i+j = n$. As have just seen, the total contribution is zero if $n<g$. Writing $n = g+\nu$ if $n \geq g$, we can rewrite~\eqref{eq:expanded} as
\[
(-1)^g \cdot \sum_{k+\ell+\nu = \mu}\; \sum_{a=0}^\nu \; \frac{T_{2g+\mu}}{T_k\cdot \ell!\cdot a!\cdot (g+\nu-a)!} \cdot \pi^*\Bigl((-\delta)^a \cdot \fTd^\inv_{\nu-a}(E^\vee_{X/S}) \cdot \fTd_k(E^\vee_{X/S}) \cdot \delta^\ell\Bigr) \cdot \bigl[e(S)\bigr]\, .
\]
Setting $c = \nu-a$, we can rewrite this as
\[
(-1)^g \cdot \sum_{k+\ell+a+c= \mu}\; \frac{T_{2g+\mu}}{T_k\cdot \ell!\cdot a!\cdot (g+c)!} \cdot \pi^*\Bigl((-\delta)^a \cdot \fTd^\inv_c(E^\vee_{X/S}) \cdot \fTd_k(E^\vee_{X/S}) \cdot \delta^\ell\Bigr) \cdot \bigl[e(S)\bigr]\, ,
\]
which by \ref{subsec:ids}\ref{binom} simplifies to
\[
(-1)^g \cdot \sum_{k+c= \mu}\; \frac{T_{2g+\mu}}{T_k\cdot (g+c)!} \cdot \pi^*\Bigl(\fTd^\inv_c(E^\vee_{X/S}) \cdot \fTd_k(E^\vee_{X/S})\Bigr) \cdot \bigl[e(S)\bigr]\, .
\]
By \ref{subsec:ids}\ref{TdTdinv} this equals $(-1)^g \cdot T_{2g} \cdot \bigl[e(S)\bigr]$ if $\mu =0$ and equals~$0$ if $\mu>0$.
\end{proof}

\section{Integral aspects of Fourier duality}

\subsection{}\label{subsec:CH(X)}
Throughout this section, $L$ is a field and $S$ is a connected quasi-projective $L$-scheme which is smooth over~$L$. Let $X$, $Y$ and~$Z$ be smooth proper $S$-schemes of finite type. 

The intersection product makes $\CH^*(X)$ a commutative graded $\CH(S)$-algebra, where the structural homomorphism is~$\pi^*$ and the identity element is the class~$[X]$. 

A class $\gamma \in \CH(X\times_S Y)$, usually referred to as a correspondence from~$X$ to~$Y$, defines a homomorphism $\gamma_* \colon \CH(X) \to \CH(Y)$ by the rule $\gamma_*(x) = \pr_{Y,*}\bigl(\gamma \cdot \pr_X^*(x)\bigr)$; here $\pr_X$ and~$\pr_Y$ are the projection maps from $X\times_S Y$ to~$X$ and~$Y$, respectively.

If $\gamma$ is a correspondence from $X$ to~$Y$ and $\delta$ is a correspondence from $Y$ to~$Z$, we define
\[
\delta \circ \gamma = \pr_{XZ,*}\bigl(\pr_{XY}^*(\gamma) \cdot \pr_{YZ}^*(\delta)\bigr)\, ,
\]
where $\pr_{XY}$ is the projection map $X\times_S Y \times_S Z \to X\times_S Y$, and similarly for the other projections. With this notation, we have the rule $(\delta\circ \gamma)_* = \delta_* \circ \gamma_*$.

\subsection{}
Let $\pi \colon X\to S$ be an abelian scheme of relative dimension $g>0$ with zero section $e \colon S\to X$. The Pontryagin product~$\star$ on~$\CH(X)$ is defined by 
\[
x\star y = m_*(x\times y)\, ,
\]
where  $m\colon X\times_S X \to X$ is the group law and $x\times y = \pr_1^*(x) \cdot \pr_2^*(y) \in \CH(X\times_S X)$ is the exterior product. 

Define
\[
\CH_i(X/S) = \CH_{d+i}(X)\, ,
\]
and let $\CH_*(X/S) = \oplus_i\; \CH_i(X/S)$. Pontryagin product makes $\CH_*(X/S)$ a commutative graded $\CH(S)$-algebra, where the ring structure on~$\CH(S)$ is the intersection product and the structural homomorphism $\CH(S) \to \CH(X/S)$ is~$e_*$. The identity element is the class $[e(S)] \in \CH_0(X/S)$ of the zero section.

\subsection{}\label{subsec:Fouriersetup}
In the rest of this section, we assume that one of the conditions in~\ref{subsec:charL} is satisfied. Let $\pi^t \colon X^t \to S$ be the dual abelian scheme, let $\cP = \cP_X$ be the Poincar\'e bundle on $X \times_S X^t$, and write $\wp = c_1(\cP) \in \CH^1(X\times_S X^t)$. Let $d = \dim(S)$.

Define $\sF_X \colon \CH(X) \to \CH(X^t)$ by $\sF_X = \gamma_{X,*}$, with 
\[
\gamma_X = \sum_{i=0}^{2g+d}\; \frac{(2g+d)!}{i!}\cdot \wp^i = (2g+d)! \cdot \ch(\cP_X)\, .
\]
Concretely, this means that for $\alpha \in \CH(X)$ we have 
\[
\sF_X(\alpha) = \pr_{X^t,*}\bigl(\gamma_X \cdot \pr_X^*(\alpha)\bigr)\; ,
\]
whose image in~$\CH(X^t;\bbQ)$ is $(2g+d)!$ times the usual Fourier transform of~$\alpha$. If the context is clear, we drop the subscript~$X$. Further, we let $\gamma^t = \gamma_{X^t}$, and we define $\sF^t = \gamma^t_* \colon \CH(X^t) \to \CH(X)$.

\subsection{}\label{subsec:N}
Define $N$ as in Lemma~\ref{lem:Nh!2} with $h = 2g+d$, so
\[
N = \begin{cases} 2 & \text{if $2g+d=3$,} \\ 2g+d+1 & \text{if $2g+d+1$ is prime}\\ 1 & \text{otherwise.} \end{cases}
\]
By that lemma and the remark in~\ref{subsec:Tm}, $T_{g+k}$ divides $N\cdot (2g+d)!^2$ for all $k\leq g+d$.

\begin{theorem}\label{thm:FtF}
Let notation and assumptions be as in~\ref{subsec:Fouriersetup}, and recall that $d = \dim(S)$ and $g = \dim(X/S)$. Let $\cH = \cH^1_\dR(X/S)$, define $\sF$ and~$\sF^t$ as above, and let $N$ be as in~\ref{subsec:N}. Then for $\alpha \in \CH(X)$ we have
\begin{equation}\label{eq:FtFformula}
N \cdot (\sF^t \circ \sF)\bigl(\alpha\bigr) \cdot \pi^*\left(\sum_{j=0}^d\; \frac{T_d}{T_j} \cdot \fTd_j(\cH) \right) = 
(-1)^g \cdot N\cdot (2g+d)!^2 \cdot T_d \cdot [-1]_X^*(\alpha)\, .
\end{equation}
\end{theorem}

(Note that $[-1]^* = [-1]_*$.)

\begin{remark}
The image of $(\sF^t \circ \sF)\bigl(\alpha\bigr)$ in $\CH(X;\bbQ)$ equals $(2g+d)!^2\cdot (\cF\circ \cF^t)\bigl(\alpha\bigr)$. The image of $\sum_{j=0}^d\; \frac{T_d}{T_j} \cdot \fTd_j(\cH)$ in $\CH(S;\bbQ)$ equals $T_d \cdot \Td(\cH)$. Formula~\eqref{eq:FtFformula} is therefore an integral version of the relation
\[
(\cF^t\circ \cF)\bigl(\alpha\bigr) \cdot \pi^*\bigl(\Td(\cH)\bigr) = (-1)^g\cdot [-1]_X^*(\alpha)\, .
\]
in $\CH(X;\bbQ)$, which holds for arbitrary abelian schemes, without any restriction on the characteristic of the base field.

Of course, the better known identity is that $\cF^t \circ \cF = (-1)^g \cdot [-1]_X^*$ in $\CH(X;\bbQ)$, which was proven by Deninger and Murre~\cite{DenMurre}. As explained in~\cite{BM-RPDM}, it is an immediate consequence of the Deninger--Murre result that $\Td(\cH) = 1$ in $\CH(S;\bbQ)$, and if  there exists an \'etale isogeny $X\to X^t$ (which is automatic in characteristic~$0$) then it follows that $c_i(\cH) = 0$ in $\CH(S;\bbQ)$ for all $i\geq 1$. For principally polarized abelian schemes, this last result was first proven by van der Geer~\cite{vdG} in a different way. We shall return to the vanishing of~$\Td(\cH)-1$ in Theorem~\ref{thm:MainLambda}.
\end{remark}

\begin{proof}[Proof of Theorem~\ref{thm:FtF}]
Define $\mu \colon X\times_S X^t\times_S X \to X\times_S X^t$ by $(x,\xi,y) \mapsto (x+y,\xi)$. By the Theorem of the Cube we have $\pi_{12}^*(\cP) \otimes \pi_{23}^*(\cP^t) \cong \mu^*\cP$ as line bundles on $X \times X^t \times X$. This gives $\pi_{12}^*(\wp) + \pi_{23}^*(\wp^t) = \mu^*(\wp)$, and hence
\begin{equation}\label{eq:pi12pi23}
\begin{aligned}
\pi_{12}^*(\gamma) \cdot \pi_{23}^*(\gamma^t) 
&= \Bigl(\sum_{i=0}^{2g+d}\; \tfrac{(2g+d)!}{i!}\cdot \pi_{12}^*(\wp)^i\Bigr) \cdot \Bigl(\sum_{j=0}^{2g+d}\; \tfrac{(2g+d)!}{j!}\cdot \pi_{23}^*(\wp^t)^j\Bigr)\\
&= \sum_{i,j=0}^\infty\; \tfrac{(2g+d)!^2}{i!j!}\cdot \pi_{12}^*(\wp)^i \cdot \pi_{23}^*(\wp^t)^j\\
&= \sum_{n=0}^\infty \; \tfrac{(2g+d)!^2}{n!} \cdot \bigl(\pi_{12}^*(\wp)^i + \pi_{23}^*(\wp^t)^j\bigr)^n \\
&= (2g+d)! \cdot \sum_{n=0}^\infty \; \tfrac{(2g+d)!}{n!} \cdot \bigl(\mu^*(\wp)\bigr)^n = (2g+d)! \cdot \mu^*(\gamma)\, .
\end{aligned}
\end{equation}
We find that $\sF^t \circ \sF$ is given by the correspondence $(2g+d)! \cdot \pi_{13,*}\mu^*(\gamma)$. Because the diagram
\begin{equation}\label{eq:mpi13}
\begin{tikzcd}[column sep=large]
X \times_S X^t \times_S X \ar[r,"\mu"] \ar[d,"\pi_{13}"'] & X \times_S X^t \ar[d,"\pr_1"] \\
X \times_S X \ar[r,"m"'] & X
\end{tikzcd}
\end{equation}
is Cartesian, this is the same as $(2g+d)! \cdot m^* \pr_{1,*}(\gamma)$. The LHS is therefore the result of applying the correspondence
\[
m^*\left(\sum_{i+j\leq g+d}\; \frac{N\cdot (2g+d)!^2\cdot T_d}{T_{g+i+j}}\cdot \frac{T_{g+i+j}}{(g+i)!\cdot T_j} \cdot \pr_{1,*}(\wp^{g+i}) \cdot \pi^*\fTd_j(\cH)\right)
\]
to $\alpha$; note that by our choice of~$N$ the factor $(N\cdot (2g+d)!^2\cdot T_d)/T_{g+i+j}$ is always an integer. By Proposition~\ref{prop:keyformula}, this correspondence is the same as
\[
(-1)^g\cdot N\cdot (2g+d)!^2\cdot T_d \cdot m^*\bigl[e(S)\bigr]\, .
\]
Because the diagram 
\[
\begin{tikzcd}[column sep=large]
X \ar[r,"{(1,-1)}"] \ar[d,"\pi"'] & X \times_S X \ar[d,"m"] \\
S \ar[r,"e"'] & X
\end{tikzcd}
\]
is Cartesian, $m^* [e(S)]$ equals the class of the graph of~$[-1]_X$ in $\CH(X\times_S X)$, and this gives the result.
\end{proof}

\begin{proposition}\label{prop:sFprops}
Let notation and assumptions be as above.
\begin{enumerate}
\item\label{sFx*y} For $x$, $y \in \CH(X)$ we have the relation
\[
(2g+d)! \cdot \sF(x\star y) = \sF(x) \cdot \sF(y)\, .
\]

\item\label{ftsF} Let $f \colon X\to Y$ be a homomorphism of abelian schemes over~$S$, with dual $f^t \colon Y^t \to X^t$. Let $g_X = \dim(X/S)$ and $g_Y = \dim(Y/S)$. If $g_X \geq g_Y$ then 
\[
f^{t,*}\bigl(\sF_X(\alpha)\bigr) = \tfrac{(2g_X+d)!}{(2g_Y+d)!} \cdot \sF_Y\bigl(f_*(\alpha)\bigr)
\]
in $\CH(Y^t)$, for all $\alpha \in \CH(X)$. If $g_X \leq g_Y$ then
\[
\tfrac{(2g_Y+d)!}{(2g_X+d)!} \cdot f^{t,*}\bigl(\sF_X(\alpha)\bigr) = \sF_Y\bigl(f_*(\alpha)\bigr)
\]
in $\CH(Y^t)$, for all $\alpha \in \CH(X)$.
\end{enumerate}
\end{proposition}

\begin{proof}
\ref{sFx*y} Recall that $x\star y = m_*(x\times y)$, so diagram~\eqref{eq:mpi13} gives $\pr_1^*(x\star y) = \mu_*\pi_{13}^*(x\times y) = \mu_*\bigl(\pi_1^*(x) \cdot \pi_3^*(y) \bigr)$. Using the projection formula and~\eqref{eq:pi12pi23}, this gives
\[
(2g+d)! \cdot \pr_1^*(x\star y) \cdot \gamma = (2g+d)! \cdot \mu_*\Bigl(\pi_1^*(x) \cdot \pi_3^*(y) \cdot \mu^*(\gamma) \Bigr) = \mu_*\Bigl(\pi_1^*(x) \cdot \pi_3^*(y) \cdot \pi_{12}^*(\gamma) \cdot \pi_{23}^*(\gamma^t) \Bigr)\, .
\]
By definition, $(2g+d)! \cdot \sF(x\star y)$ is the image under $\pr_{2,*}$ of this expression. Note that $\pr_2 \circ \mu = \pi_2 = \pr_2 \circ \pi_{12}$ and that $\pi_1^*(x) = \pi_{12}^*\pr_1^*(x)$ and $\pi_3^*(y) = \pi_{23}^*\pr_2^*(y)$. This gives
\[
(2g+d)! \cdot \sF(x\star y) = \pr_{2,*}\Bigl(\pr_1^*(x) \cdot \gamma \cdot \pi_{12,*}\pi_{23}^*\bigl(\pr_2^*(y) \cdot \gamma^t\bigr)\Bigr)\, .
\]
Because the diagram
\[
\begin{tikzcd}
X\times X^t \times X \ar[r,"\pi_{12}"] \ar[d,"\pi_{23}"'] & X\times X^t \ar[d,"\pr_2"]\\
X^t \times X \ar[r,"\pr_1"] & X^t
\end{tikzcd}
\]
is Cartesian, it follows that 
\[
(2g+d)! \cdot \sF(x\star y) = \pr_{2,*}\Bigl(\pr_1^*(x) \cdot \gamma \cdot \pr_2^*\pr_{1,*}\bigl(\pr_2^*(y) \cdot \gamma^t\bigr)\Bigr) = \pr_{2,*}(\pr_1^*(x) \cdot \gamma) \cdot \pr_{1,*}(\pr_2^*(y) \cdot \gamma^t) = \sF(x) \cdot \sF(y)\, ,
\]
as claimed.

\ref{ftsF} We have $(f\times \id)^*\cP_Y \cong (\id\times f^t)^* \cP_X$ as line bundles on $X\times_S Y^t$. This gives the relation
\[
\begin{cases}
(\id \times f^t)^* \gamma_X = \tfrac{(2g_X+d)!}{(2g_Y+d)!} \cdot (f\times \id)^* \gamma_Y & \text{if $g_X \geq g_Y$,}\\
\tfrac{(2g_Y+d)!}{(2g_X+d)!} \cdot (\id \times f^t)^* \gamma_X = (f\times \id)^* \gamma_Y & \text{if $g_X \leq g_Y$.}
\end{cases}
\]
In the first case, write $q = \frac{(2g_X+d)!}{(2g_Y+d)!}$. The diagram
\[
\begin{tikzcd}[row sep=small]
&& X\times_S Y^t \ar[dl,"\id\times f^t"'] \ar[dr,"\pr_2"'] \ar[rr,"f\times \id"] && Y \times_S Y^t \ar[dl,"\pr_2"]\\
& X\times_S X^t \ar[dl,"\pr_1"] \ar[dr,"\pr_2"'] && Y^t\ar[dl,"f^t"]\\
X && X^t &
\end{tikzcd}
\]
then gives
\begin{multline*}
f^{t,*}\bigl(\sF_X(\alpha)\bigr) = f^{t,*} \pr_{2,*}\bigl(\pr_1^*(\alpha) \cdot \gamma_X\bigr) = \pr_{2,*}\Bigl(\pr_1^*(\alpha) \cdot (\id\times f^t)^*(\gamma_X)\Bigr) =\\
q \cdot \pr_{2,*} (f\times \id)_*\Bigl(\pr_1^*(\alpha)\cdot  (f\times \id)^*(\gamma_Y) \Bigr) = q \cdot \pr_{2,*}\Bigl((f\times \id)_*\pr_1^*(\alpha) \cdot \gamma_Y\Bigr) = \\
q \cdot \pr_{2,*}\Bigl(\pr_1^*\bigl(f_*(\alpha)\bigr) \cdot \gamma_Y \Bigr) = q \cdot  \sF_Y\bigl(f_*(\alpha)\bigr)\, .
\end{multline*}
The argument in the second case is entirely similar.
\end{proof}

\section{Fourier theory after inverting \texorpdfstring{$(2g+d+1)!$}{2g}}\label{sec:Fourier}

In all of this section, we retain the notation and assumptions of the previous section, and we assume that for all abelian schemes that we consider, at least one of the conditions in~\ref{subsec:charL} is satisfied. 

\subsection{}
For $\pi\colon X\to S$ an abelian scheme as before, define
\[
\Lambda = \bbZ\Bigl[\frac{1}{(2g+d+1)!}\Bigr]\, .
\]
We can then define
\[
\cF = \cF_X \colon \CH(X;\Lambda) \to \CH(X^t;\Lambda)\quad\text{by}\quad
\cF(\alpha) = (2g+d)!^{-1} \cdot \sF_X(\alpha)\, .
\]
This refines the classical Fourier transform. We write $\cF^t$ for~$\cF_{X^t}$. 

It follows from Theorem~\ref{thm:FtF} that for $\alpha \in \CH(X;\Lambda)$ we have
\begin{equation}\label{eq:FtFTdH}
(\cF^t\circ \cF)\bigl(\alpha\bigr) \cdot \pi^*\Td(\cH) = (-1)^g \cdot [-1]^*(\alpha)\, ,
\end{equation}
in $\CH(X;\Lambda)$, where we recall that $\cH = \cH^1_\dR(X/S)$, and where $\Td(\cH)$ is the Todd class of~$\cH$ as an element of~$\CH(S;\Lambda)$. (It follows from Lemma~\ref{lem:Nh!2} that $T_d \in \Lambda^*$, so $\Td(\cH) \in \CH(S;\Lambda)$ is well-defined.) In particular, for $\alpha \in \CH^i(X;\Lambda)$ we have
\begin{equation}\label{eq:FtFmodCH>i}
(\cF^t\circ \cF)\bigl(\alpha\bigr) = (-1)^g \cdot [-1]^*(\alpha) + \beta\, ,\quad \text{for some $\beta \in \CH^{>i}(X;\Lambda)$.}
\end{equation}
As we shall see below, we in fact have $\Td(\cH) = 1$ and $\beta = 0$. (The pattern of the arguments is the same as in~\cite{DenMurre}.) 
\medskip

Our next goal is to define a version of Beauville's decomposition (see~\cite{BeauvChow}) with coefficients in the ring~$\Lambda$. We start with two lemmas that are inspired by Beauville's calculations in~\cite{BeauvFourier}.

\begin{lemma}\label{lem:n*x=n2i-sx}
Let $x$ be an element of~$\CH^i(X;\Lambda)$. Suppose there exists an integer~$s$ with $i-2g\leq s \leq 2i$ such that $[n]^*(x) = n^{2i-s} \cdot x$ for all $n \in \bbZ$. Then $\cF(x) \in \CH^{g-i+s}(X^t;\Lambda)$.
\end{lemma}

\begin{proof}
Throughout we work with coefficients in the ring~$\Lambda$. We may assume $x\neq 0$, so that $0\leq i\leq g+d$. Write
\[
j_{\min} = \max\{0,g-i\}\, ,\quad j_{\max} = 2g+d-i\, .
\]
For $j \in \{j_{\min},\ldots,j_{\max}\}$, define 
\begin{equation}\label{eq:yjdef}
y_j = \pr_{X^t,*}\Bigl(\pr_X^*(x) \cdot \frac{\wp^j}{j!}\Bigr)\, .
\end{equation}
(Note that $j! \in \Lambda^*$.) Then $y_j \in \CH^{i+j-g}(X^t;\Lambda)$ and 
\[
\cF(x) = \sum_{j=j_{\min}}^{j_{\max}}\; y_j\, .
\] 
Moreover, $[n]^*(y_j) = n^j\cdot y_j$ for all $n \in \bbZ$. (Use that $[n]^* \circ \pr_{X^t,*} = \pr_{X^t,*} \circ (\id_X \times [n])^*$ together with the fact that $(\id_X \times [n])^*(\wp) = n\cdot \wp$.) The assumption that $[n]^*(x) = n^{2i-s} \cdot x$ gives that $n^{2i-s} \cdot [n]_*(x) = [n]_*[n]^*(x) = n^{2g}\cdot x$ for all~$n$. By Proposition~\ref{prop:sFprops}\ref{ftsF} it follows that $n^{2i-s} \cdot [n]^*\cF(x) = n^{2i-s} \cdot \cF\bigl([n]_*(x) \bigr) = n^{2g}\cdot \cF(x)$ in~$\CH(X^t;\Lambda)$. Hence,
\begin{equation}\label{eq:systeqs}
\sum_{j=j_{\min}}^{j_{\max}}\; \bigl(n^{2g} - n^{2i+j-s}\bigr) \cdot y_j = 0
\end{equation}
in $\CH(X^t;\Lambda)$, for all $n \in \bbZ$. Writing $J = \{j_{\min},\ldots,j_{\max}\} \setminus \{2g-2i+s\}$, this is an infinite system of linear equations in the~$y_j$, for $j \in J$, and the claim is that there is only the trivial solution $y_j = 0$ for all $j \in J$. (This then gives the lemma because it implies that $\cF(x) = y_{2g-2i+s} \in \CH^{g-i+s}(X^t;\Lambda)$.) Because $\Lambda$ is a PID and $Y = \sum_{j\in J}\; \Lambda \cdot y_j \subset \CH(X^t;\Lambda)$ is clearly finitely generated, it suffices to show that the system of linear equations~\eqref{eq:systeqs} only has the trivial solution in $Y\otimes \bbQ$ and in $Y\otimes \bbF_p$ for $p>2g+d+1$. Suppose there is nontrivial solution $(y_j)_{j\in J}$ to~\eqref{eq:systeqs} in~$Y\otimes \bbQ$. Then the set of all $n\in \bbQ$ for which \eqref{eq:systeqs} holds in $Y\otimes \bbQ$ is Zariski closed in~$\bbQ$ and contains~$\bbZ$; hence \eqref{eq:systeqs} holds for all $n \in \bbQ$; however, this contradicts the linear independence of characters (see~\cite{Stacks}, Lemma~0CKL). Similarly, if $p$ is a prime number with $p>2g+d+1$, the characters $n \mapsto n^{2i+j-s}$ for $j \in J$ and the character $n \mapsto n^{2g}$ are all distinct on~$\bbF_p$. (Here the point is that we cannot have $2i+j-s = 2g$ by definition of~$J$, and that the assumptions $p>2g+d+1$ and $i-2g\leq s$ together imply that $2i+j-s < 2g+(p-1)$.) So again it follows from the linear independence of characters that the system of equations~\eqref{eq:systeqs} has no nontrivial solution $(y_j)_{j\in J}$ in~$Y\otimes \bbF_p$.
\end{proof}

It follows from the proof that if $0 \neq x\in \CH^i(X;\Lambda)$ satisfies $[n]^*(x) = n^{2i-s} \cdot x$ for some~$i$ with $i-2g\leq s \leq 2i$ then necessarily 
\begin{equation}\label{eq:srange}
\max\{i-g,2i-2g\} \leq s \leq \min\{i+d,2i\}\, .
\end{equation}

\begin{lemma}\label{lem:(a)-(d)}
Let $x \in \CH^i(X;\Lambda)$ and let $s$ be an integer that satisfies~\eqref{eq:srange}. Then the following properties are equivalent.
\begin{enumerate}[label=\textup{(\alph*)}]
\item\label{cond(a)} $[n]^*(x) = n^{2i-s}\cdot x$ for all $n \in \bbZ$;

\item\label{cond(c)} $\cF(x) \in \CH^{g-i+s}(X^t;\Lambda)$;

\item\label{cond(d)} $\cF(x) \in \CH^{g-i+s}(X^t;\Lambda)$ and $[n]^*\bigl(\cF(x)\bigr) = n^{2g-2i+s} \cdot \cF(x)$ for all $n\in \bbZ$.
\end{enumerate}
Moreover, these conditions imply that $(\cF^t \circ \cF)\bigl(x\bigr) = (-1)^g\cdot [-1]^*(x)$.
\end{lemma}

\begin{proof}
If $x = 0$ there is nothing to prove, so we may assume $x\neq 0$. As in the proof of the previous lemma, write $\cF(x) = \sum_j\; y_j$ for $j = j_{\min},\ldots,j_{\max}$, with $y_j$ as in~\eqref{eq:yjdef}. Then \ref{cond(c)} and~\ref{cond(d)} are both equivalent to $\cF(x) = y_{2g-2i+s}$. The implication \ref{cond(a)} $\Rightarrow$~\ref{cond(c)} is the previous lemma. Next suppose \ref{cond(d)} holds. Applying the (now proven) implication \ref{cond(a)} $\Rightarrow$~\ref{cond(d)} to the element~$\cF(x)$, it follows that $(\cF^t \circ \cF)\bigl(x\bigr) \in \CH^i(X;\Lambda)$ and that $[n]^*\bigl((\cF^t \circ \cF)(x)\bigr) = n^{2i-s} \cdot (\cF^t \circ \cF)\bigl(x\bigr)$. Comparing this with~\eqref{eq:FtFmodCH>i}, it follows that $(\cF^t \circ \cF)\bigl(x\bigr) = (-1)^g\cdot [-1]^*(x)$, and hence that \ref{cond(a)} holds.
\end{proof}

\begin{definition}\label{def:CHi(s)}
For $s\in \bbZ$ satisfying~\eqref{eq:srange}, let 
\[
\CH^i_{(s)}(X;\Lambda) \subset \CH^i(X;\Lambda)
\]
be the $\Lambda$-submodule of elements~$x$ that satisfy the above equivalent conditions \ref{cond(a)}--\ref{cond(d)}. If $s$ does not satisfy~\eqref{eq:srange}, let $\CH^i_{(s)}(X;\Lambda)=0$.

We also define 
\[
\CH_{i,(s)}(X/S;\Lambda) = \CH^{g-i}_{(s)}(X;\Lambda)\, .
\]

Note that the subspaces $\CH^i_{(s)}(X;\Lambda)$, and hence also the $\CH_{i,(s)}(X/S;\Lambda)$, are stable under all operators~$[n]^*$ and~$[n]_*$.
\end{definition}

\begin{theorem}\label{thm:MainLambda}
Let notation and assumptions be as above.
\begin{enumerate}
\item\label{TdHFtF} We have $\Td(\cH) = 1$ in $\CH(S;\Lambda)$, and $\cF^t \circ \cF = (-1)^g\cdot [-1]^*$.

\item\label{Fa*b} For $\alpha$, $\beta \in \CH(X;\Lambda)$ we have the relations
\[
\cF(\alpha \star \beta) = \cF(\alpha) \cdot \cF(\beta)\, ,\qquad
\cF(\alpha \cdot \beta) = (-1)^g \cdot \cF(\alpha) \star \cF(\beta)\, .
\]

\item\label{BeauvDec} We have a decomposition
\[
\CH^i(X;\Lambda) = \bigoplus_{s=\max\{i-g,2i-2g\}}^{\min\{i+d,2i\}}\; \CH^i_{(s)}(X;\Lambda)\, ,
\]
and $\cF$ restricts to bijections $\CH^i_{(s)}(X;\Lambda) \isomarrow \CH^{g-i+s}_{(s)}(X^t;\Lambda)$.

\item\label{bigraded} The intersection product restricts to maps
\[
\CH^i_{(s)}(X;\Lambda) \times \CH^j_{(t)}(X;\Lambda) \to \CH^{i+j}_{(s+t)}(X;\Lambda)\, ,
\]
and the Pontryagin product restricts to maps
\[
\CH_{i,(s)}(X/S;\Lambda) \times \CH_{j,(t)}(X/S;\Lambda) \to \CH_{i+j,(s+t)}(X/S;\Lambda)\, .
\]
 
\item\label{homom} If $Y/S$ is an abelian scheme of relative dimension $h\leq g$ and $f \colon X\to Y$ is a homomorphism with dual $f^t\colon Y^t \to X^t$ then
\[
(f^t)^* \circ \cF_X = \cF_Y \circ f_* \colon \CH(X;\Lambda) \to \CH(Y^t;\Lambda)
\]
and
\[
\cF_X \circ f^* = (-1)^{g-h}\cdot f^t_*\circ \cF_Y \colon \CH(Y;\Lambda) \to \CH(X^t;\Lambda)\, .
\]

\end{enumerate}
\end{theorem}

\begin{proof}
We first prove~\ref{BeauvDec} together with the second assertion of~\ref{TdHFtF}. Using property~\ref{lem:(a)-(d)}\ref{cond(c)}, it is immediate that the subspaces $\CH^i_{(s)}(X;\Lambda)$ are linearly independent. If $x \in \CH^i_{(s)}(X;\Lambda)$ then it follows from property~\ref{lem:(a)-(d)}\ref{cond(d)} that $\cF(x) \in \CH^{g-i+s}_{(s)}(X;\Lambda)$. Next, let $x \in \CH^i(X;\Lambda)$ be arbitrary and write $\cF(x) = \sum_j\; y_j$ as before. Then $y_j \in \CH^{i+j-g}_{(2i-2g+j)}(X^t;\Lambda)$, and by Lemma~\ref{lem:(a)-(d)} applied to~$X^t$ it follows that $\cF^t(y_j) \in \CH^i_{(2i-2g+j)}$. Because $(\cF^t\circ \cF)\bigl(x\bigr) = \sum \cF^t(y_j)$, it follows from \eqref{eq:FtFmodCH>i} that $(-1)^g \cdot [-1]^*(x) = (\cF^t\circ \cF)\bigl(x\bigr)$ and hence that $x$ lies in the span of the subspaces~$\CH^i_{(s)}(X;\Lambda)$.

The first assertion of~\ref{TdHFtF} now follows by taking $\alpha = [X]$ in~\eqref{eq:FtFTdH}. The first assertion of~\ref{homom} follows from Proposition~\ref{prop:sFprops}\ref{ftsF}, and the second then follows by using~\ref{TdHFtF} (both for $X$ and for~$Y$). Similarly, the first assertion of~\ref{Fa*b} follows from Proposition~\ref{prop:sFprops}\ref{sFx*y}, and the second then follows by duality.

For~\ref{bigraded}, finally, suppose $0 \neq \alpha \in \CH^i_{(s)}(X;\Lambda)$ and $0\neq \beta \in \CH^j_{(t)}(X;\Lambda)$. It is clear that $[n]^*(\alpha \cdot \beta) = n^{2(i+j)-(s+t)} \cdot (\alpha \cdot \beta)$ and that $(i+j)-2g \leq s+t \leq 2(i+j)$. By the remark after Lemma~\ref{lem:n*x=n2i-sx}, if $\alpha \cdot \beta \neq 0$ then we have $\max\{(i+j)-g,2(i+j)-2g\} \leq s+t \leq \min\{i+j+d,2(i+j)\}$, and in either case we find that $\alpha \cdot \beta \in \CH^{i+j}_{(s+t)}(X;\Lambda)$. The second assertion follows by using \ref{TdHFtF} and~\ref{Fa*b} (cf.\ the next remark).
\end{proof}

\begin{remark}
It readily follows from what we have shown that \ref{cond(a)}--\ref{cond(d)} in Lemma~\ref{lem:(a)-(d)} are also equivalent to the condition that $[n]_*(x) = n^{2g-2i+s}\cdot x$ for all $n \in \bbZ$. This means that 
\begin{equation}\label{eq:CHisn_*}
\CH_{i,(s)}(X/S;\Lambda) = \bigl\{\alpha \in \CH_i(X/S;\Lambda) \bigm| [n]_*(\alpha) = n^{2i+s} \cdot \alpha \quad \text{for all $n\in \bbZ$}\bigr\}\, .
\end{equation}
\end{remark}

\subsection{}\label{subsec:idealI}
Recall that the Pontryagin product makes $\CH_\bullet(X/S;\Lambda)$ into an algebra over $\CH(S;\Lambda)$, where the ring structure on~$\CH(S;\Lambda)$ is the intersection product and the structural homomorpism is $e_*\colon \CH(S;\Lambda) \to \CH_\bullet(X/S;\Lambda)$. Because $\pi_* \circ e_*$ is the identity on~$\CH(S;\Lambda)$ and $e_* \circ \pi_* = [0]_*$ on~$\CH(X;\Lambda)$, it follows from~\eqref{eq:CHisn_*}, taking $n=0$, that $e_*$ induces isomorphisms $\CH_j(S;\Lambda) \isomarrow \CH_{j-d,(2d-2j)}(X/S;\Lambda)$, with inverse given by~$\pi_*$. In particular, $\CH_{0,(0)}(X/S;\Lambda) = \Lambda \cdot \bigl[e(S)\bigr]$. Further, $\pi_*$ is zero on $\CH_{i,(s)}(X/S;\Lambda)$ whenever $2i+s \neq 0$, for in that case $e_*\circ \pi_* = [0]_*$ is zero, whereas $e_*$ is injective.

It is immediate from Theorem~\ref{thm:MainLambda}\ref{bigraded} that the space $\CH_0(X/S;\Lambda)$ of relative $0$-cycles is a $\Lambda$-subalgebra of $\CH_\bullet(X/S;\Lambda)$ for the Pontryagin ring structure, which is graded by
\[
\CH_0(X/S;\Lambda) = \bigoplus_{s=0}^{\min\{g+d,2g\}}\; \CH_{0,(s)}(X/S;\Lambda)\, .
\]
We define the augmentation ideal $I\subset \CH_0(X/S;\Lambda)$ to be the kernel of~$\pi_*$, i.e.,
\[
I = \Ker(\pi_*) = \bigoplus_{s>0}\; \CH_{0,(s)}(X/S;\Lambda)\, .
\]
(For the second equality, again take $n=0$ in~\eqref{eq:CHisn_*}.) Parts~\ref{BeauvDec} and~\ref{bigraded} of Theorem~\ref{thm:MainLambda} give the following result, which is the version with $\Lambda$-coefficients of a result that goes back to Bloch~\cite{Bloch}, Theorem~0.1.

\begin{corollary}\label{cor:Istarnu+1}
With notation as above, let $\nu = \min\{g+d,2g\}$. Then $I^{\star (\nu+1)} = 0$.
\end{corollary}

\section{Some results about integral representations of~\texorpdfstring{$\fsl_2$}{sl}}

This section is an interlude on representations of the Lie algebra~$\fsl_2$. Proposition~\ref{prop:sl2rep} will be applied to Chow groups in the next section. These results are probably known, but we have not been able to locate them in the literature in the form that we require.

\subsection{}\label{subsec:SymnSt}
Let $g$ be a positive integer. In this section, $\Lambda$ will be any ring of the form $\Lambda = \bbZ\bigl[\frac{1}{n}\bigr]$ with $n$ divisible by~$(2g)!$. By $\fsl_2$ we mean the Lie algebra of that name over the ring~$\Lambda$. The underlying module is free of rank~$3$ with generators that we call $e$, $h$ and~$f$, and the Lie bracket is given by
\[
[h,e] = 2e\, ,\quad [e,f] = h\, ,\quad [h,f] = -2f\, .
\]
We shall be interested in representations $\rho\colon \fsl_2 \to \fgl(V)$, for $V$ a $\Lambda$-module, such that there is a grading $V = \oplus_{i \in \bbZ}\; V_i$ with the property that $\rho(h)$ is multiplication by~$i$ on~$V_i$. In that case, the first and third commutation relations just say that $\rho(e)$ is an endomorphism of degree~$+2$ and $\rho(f)$ is an endomorphism of degree~$-2$ with respect to the given grading. If it is clear from the context what we mean, we write $e$, $h$ and~$f$ instead of $\rho(e)$, $\rho(h)$ and~$\rho(f)$, and we call~$V$ an $\fsl_2$-module. Note that in general we do not assume the underlying $\Lambda$-module to be of finite type or to be torsion-free. 

If $V$ is an $\fsl_2$-module, we denote by~$V^*$ the $\fsl_2$-module that has the same underlying $\Lambda$-module, with $\fsl_2$-action given by the rules
\[
e_{V^*} = f_V\, ,\quad h_{V^*} = -h_V\, ,\quad f_{V^*} = e_V\, .
\]

\subsection{}
Let $\St$ denote the standard representation of~$\fsl_2$ over~$\Lambda$, by which we mean that $\St = \Lambda \cdot x_{-1} \oplus \Lambda\cdot x_1$ as a module over~$\Lambda$, with action of~$\fsl_2$ given, with respect to the basis~$\{x_{-1},x_1\}$ by
\[
e = \begin{pmatrix}0 & 0\\ 1 & 0 \end{pmatrix}\, ,\qquad
h = \begin{pmatrix}-1 & 0\\ 0 & 1 \end{pmatrix}\, ,\qquad
f = \begin{pmatrix}0 & 1\\ 0 & 0 \end{pmatrix}\, .
\]
We are interested in the symmetric powers~$\Sym^n(\St)$ of this representation, for $n\leq g$. (This restriction is crucial in many of our calculations.) Explicitly, the underlying $\Lambda$-module of~$\Sym^n(\St)$ is free of rank~$n+1$ with basis $x_{-n}, x_{-n+2},\ldots,x_n$, and the $\fsl_2$-action is given by the rules
\begin{equation}\label{eq:SymnExpl}
e(x_{-n+2i}) = (n-i) \cdot x_{-n+2i+2}\, ,\qquad 
h(x_i) = i\cdot x_i\, ,\qquad
f(x_{-n+2i}) = i\cdot x_{-n+2i-2}\, .
\end{equation}

\subsection{}\label{subsec:V=sumVi}
Let $V = \oplus_{i=-g}^g\; V_i$ where the $V_i$ are $\Lambda$-modules. Assume we have a representation of~$\fsl_2$ on~$V$ such that $h \in \fsl_2$ is multiplication by~$i$ on~$V_i$. For $j \in \{-g,\ldots,g\}$ define
\[
V_j[e] = \bigl\{v\in V_j\bigm| e(v) = 0\bigr\}\, ,\quad
V_j[f] = \bigl\{v\in V_j\bigm| f(v) = 0\bigr\}\, .
\]

\begin{lemma}\label{lem:flek}
Suppose $v \in V_{-n}[f]$ for some integer~$n$ with $0\leq n\leq g$.
\begin{enumerate}
\item\label{flek1} Let $k$ and~$\ell$ be integers with $0 \leq k \leq n$ and $\ell\geq 0$. Then
\[
f^\ell e^k(v) = \begin{cases}
\frac{(n-k+\ell)!\, k!}{(n-k)!\, (k-\ell)!}\cdot e^{k-\ell}(v) & \text{if $\ell\leq k$,}\\
0 & \text{if $\ell> k$.}
\end{cases}
\]

\item\label{flek2} We have $e^{n+1}(v) = 0$.

\item\label{flek3} For $1\leq n\leq g$ we have $V_n[f] = 0$.
\end{enumerate}
\end{lemma}

\begin{proof} \ref{flek1} We argue by double induction. The case $k=0$ is trivial. Assume then that $1\leq k\leq n$ and that the assertion is true for smaller values of~$k$. For $\ell=0$ the assertion is again trivial. If $1\leq \ell \leq k-1$, we have
\[
f^\ell h e^{k-1}(v) = (-n+2k-2) \cdot f^\ell e^{k-1}(v) = (-n+2k-2) \cdot \frac{(n-k+1+\ell)!\, (k-1)!}{(n-k+1)!\, (k-1-\ell)!} \cdot e^{k-1-\ell}(v)\, .
\]
On the other hand, $h= [e,f]$ gives that
\begin{align*}
f^\ell h e^{k-1}(v) &= f^\ell efe^{k-1}(v) - f^{\ell+1}e^k(v)\\
&= \frac{(n-k+1+1)!\, (k-1)!}{(n-k+1)!\, (k-2)!}\cdot f^\ell e^{k-1}(v)  - f^{\ell+1}e^k(v)\\
&= (n-k+2)(k-1) \cdot \frac{(n-k+1+\ell)!\, (k-1)!}{(n-k+1)!\, (k-1-\ell)!} \cdot e^{k-1-\ell}(v) - f^{\ell+1}e^k(v)\, .
\end{align*}
Comparing the two answers, we find $f^{\ell+1}e^k(v) = C\cdot e^{k-1-\ell}(v)$ with
\[
C = \bigl\{(n-k+2)(k-1) - (-n+2k-2)\bigr\} \cdot \frac{(n-k+1+\ell)!\, (k-1)!}{(n-k+1)!\, (k-1-\ell)!} = \frac{(n-k+1+\ell)!\, k!}{(n-k)!\, (k-1-\ell)!}\, ,
\]
as desired, and by induction we get the stated formula for all $\ell\leq k$. Finally, if $\ell = k$ the same method gives that $f^{k+1}e^k(v) = 0$.

\ref{flek2} First we show that 
\begin{equation}\label{eq:fen+1+k}
fe^{n+1+k}(v) = -k(n+k+1) \cdot e^{n+k}(v)
\end{equation}
for all $k\geq 0$. By part~\ref{flek1} we have $fe^n(v) = n\cdot e^{n-1}(v)$, and then
\[
n\cdot e^n(v) = he^n(v) = efe^n(v) - fe^{n+1}(v) = n \cdot e^n(v) - fe^{n+1}(v)
\]
gives the desired relation for $k=0$. (One checks that this also works if $n=0$; in this case $fe^0(v) = 0$.) We continue by induction, which is achieved by using that
\[
(n+2+2k) \cdot e^{n+1+k}(e) = he^{n+1+k}(v) = efe^{n+1+k}(v) - fe^{n+2+k}(v) = -k(n+k+1) \cdot e^{n+1+k} -  fe^{n+2+k}(v).
\]
This gives~\eqref{eq:fen+1+k}.

If $v\neq 0$, our assumptions on~$V$ imply that there exists some $N\geq 0$ such that $e^N(v)\neq 0$ and $e^{N+1}(v) = 0$. It follows from part~\ref{flek1}, taking $k=\ell=n$, that we must have $N\geq n$. (Note that $n! \in \Lambda^*$.) By \eqref{eq:fen+1+k}, if $N>n$ then $e^{N+1}(v) = 0$ implies that $e^N(v)=0$. Hence we must have $N=n$, so $e^{n+1}(v) = 0$.

\ref{flek3} Suppose $v \in V_n[f]$. With the same arguments as in the proof of~\ref{flek1}, we find that for all $0\leq \ell\leq k\leq \lfloor\frac{g+2-n}{2}\rfloor$ we have
\[
f^\ell e^k(v) = (-1)^\ell \cdot \frac{(n+k-1)!\, k!}{(n+k-1-\ell)!\, (k-\ell)!} \cdot e^{k-\ell}(v)\, .
\]
Now take $k=\ell = \lfloor\frac{g+2-n}{2}\rfloor$. Then $e^k(v) = 0$ because $e^k(v) \in V_{n+2k}$ but $n+2k > g$. On the other hand, $f^ke^k(v)= C\cdot v$ with $C = \frac{(n+k-1)!\, k!}{(n-1)!} \in \Lambda^*$. It follows that $v=0$.
\end{proof}

\begin{remark}
By applying the lemma to~$V^*$ instead of~$V$ we get analogous conclusions for $v\in V_n[e]$, as well as the conclusion that $V_{-n}[e] = 0$ for $1\leq n\leq g$.
\end{remark}

\begin{lemma}\label{lem:homog}
Let $V = \oplus_{i=-g}^g\; V_i$ be an $\fsl_2$-module as in~\ref{subsec:V=sumVi}, let $W \subset V$ be an $\fsl_2$-submodule, and for $i \in \{-g,\ldots,g\}$ define $W_i = W\cap V_i$. Then $W = \oplus_{i=-g}^g\; W_i$.
\end{lemma}

\begin{proof}
Let $w$ be an element of~$W$, and write it as $w = \sum_{i=-g}^g\; w_i$ with $w_i \in V_i$. Then $\sum_{i=-g}^g\; i^n \cdot w_i = h^n(w) \in W$ for every $n\geq 0$. By our assumption on~$\Lambda$ we have $(2g)! \in \Lambda^*$, so
\[
\det\bigl(i^n\bigr)_{\substack{i=-g,\ldots,g\\n=0,\ldots,2g}} = \prod_{-g\leq i<j\leq g}\; (j-i)\, ,
\]
is a unit in~$\Lambda$. It follows that $w_i \in W$ for all~$i$.
\end{proof}

\begin{proposition}\label{prop:sl2rep}
Let $V = \oplus_{i=-g}^g\; V_i$ be an $\fsl_2$-module as in~\ref{subsec:V=sumVi}.
\begin{enumerate}
\item\label{sl2rep1} For $n \in \{0,\ldots,g\}$, the map $v \mapsto (n!)^{-1} \cdot e^n(v)$ defines an isomorphism of $\Lambda$-modules $V_{-n}[f] \isomarrow V_n[e]$, with inverse given by $v \mapsto (n!)^{-1} \cdot f^n(v)$. In particular, for all $i\leq n$ the maps $e^i \colon V_{-n}[f] \to V$ and $f^i \colon V_n[e]\to V$ are injective.

\item\label{sl2rep2} For all $i \in \{0,\ldots,n\}$ we have $e^iV_{-n}[f] = f^{n-i}V_n[e]$ as $\Lambda$-submodules of~$V_{-n+2i}$.

\item\label{sl2rep3} Define
\[
M_n = \sum_{i=0}^n\; e^iV_{-n}[f] \subset V\, .
\]
Then $M_n$ is the direct sum of its subspaces~$e^iV_{-n}[f]$, and $M_n$ is an $\fsl_2$-submodule of~$V$.

\item\label{sl2rep4} Let $x_{-n}, x_{-n+2},\ldots,x_n$ be the basis of~$\Sym^n(\St)$ as in~\ref{subsec:SymnSt}. Then the homomorphism
\[
\Sym^n(\St) \otimes_\Lambda V_{-n}[f] \to M_n
\]
given by $x_{-n+2i} \otimes v\mapsto \frac{(n-i)!}{n!} \cdot e^i(v)$ is an isomorphism of $\fsl_2$-modules.

\item\label{sl2rep5} The natural map
\[
\phi \colon \bigoplus_{n=0}^g\; M_n \to V
\]
is an isomorphism of $\fsl_2$-modules.
\end{enumerate}
\end{proposition}

\begin{proof} Assertions \ref{sl2rep1} and~\ref{sl2rep2} readily follow from Lemma~\ref{lem:flek}, applied to $V$ and~$V^*$. In~\ref{sl2rep3}, it is immediate from the fact that $e^iV_{-n}[f] \subset V_{-n+2i}$ that the sum is direct. Further, it is clear from the construction together with~\ref{sl2rep1} and~\ref{sl2rep2} that $M_n \subset V$ is stable under the operators $e$, $h$ and~$f$.

\ref{sl2rep4} It follows from~\ref{sl2rep3} that the given map is an isomorphism of $\Lambda$-modules, and it only remains to be shown that it is $\fsl_2$-equivariant. Using~\eqref{eq:SymnExpl}, compatibility with the operators $e$ and~$h$ is clear, and compatibility with~$f$ follows from the case $\ell=1$ of Lemma~\ref{lem:flek}\ref{flek1}.

\ref{sl2rep5} By Lemma~\ref{lem:homog}, applied to the representation $\oplus M_n$, the kernel $K = \Ker(\phi)$ is homogeneous, i.e., $K = \oplus_{i=-g}^g\; K_i$. Suppose $K \neq 0$. Let $s$ be the smallest index in $\{-g,\ldots,g\}$ with $K_s \neq 0$, and choose $0\neq v\in K_s$. Write $v = \sum_{n=0}^g\; v_n$ with $v_n \in M_n$. By our choice of~$s$ we have $f(v) = 0$, and hence $v_n \in M_n[f]$ for all~$n$. But it follows from Lemma~\ref{lem:flek}, applied to the representation~$M_n$, that $M_n[f] = (M_n)_{-n}$. Hence we must have $s \leq 0$, say $s = -r$, and $v_n = 0$ for all $n \neq r$. This gives that $v = v_r \in (M_r)_{-r} = V_{-r}[f]$. On this space, $\phi$ is just the inclusion map $V_{-r}[f] \hookrightarrow V$, so we get a contradiction, i.e., $\phi$ is injective.

Similarly, the image $\Image(\phi) \subset V$ is homogeneous. Suppose $\phi$ is not surjective. Let $s$ be the smallest index in $\{-g,\ldots,g\}$ with $\Image(\phi)_s \neq V_s$, and choose $v\in V_s\setminus \Image(\phi)_s$. Then $f(v) \in \Image(\phi)_{s-2}$, so we can write $f(v) = \sum w_n$ with $w_n \in M_n$, where the sum runs over the integers $n \in \{0,\ldots,g\}$ with $n\equiv s \bmod{2}$ and $n\geq 2-s$. (In all other cases, $(M_n)_{s-2} = 0$.) By part~\ref{sl2rep4}, we can write~$w_n$ as 
\[
w_n = e^{\frac{s-2+n}{2}}(y_{-n})  
\]
for some $y_{-n} \in (M_n)_{-n} = V_{-n}[f]$. As clearly $V_{-n}[f] \subset \Image(\phi)$, it follows that $w_n \in \Image(\phi)$. Using Lemma~\ref{lem:flek}\ref{flek1}, we find constants $c_n \in \Lambda^*$ such that $c_n \cdot fe(w_n) = c_n \cdot fe^{1+\frac{s-2+n}{2}}(y_{-n})$ equals~$w_n$, for all~$n$. Then $z = \sum_n\; c_n \cdot e(w_n) \in \Image(\phi)$, and $f(v-z) = 0$. But this means that $v-z \in V_s[f]$. If $s\geq 1$ then $V_s[f] = 0$ by Lemma~\ref{lem:flek}\ref{flek3}, so in this case we get $v = z \in \Image(\phi)$, contradiction. If $s = -r \leq 0$ then we find that $v-z \in V_{-r}[f]$, which is contained in the image of~$\phi$. So again in this case it follows that $v = z + (v-z) \in \Image(\phi)$, contradiction.
\end{proof}

\section{The \texorpdfstring{$\fsl_2$}{sl}-action on the Chow ring}\label{sec:sl2}

In all of this section, the same notation and assumptions as in Section~\ref{sec:Fourier} are in force, including the assumption that one of the conditions in~\ref{subsec:charL} is satisfied. Further, we choose a polarization $\theta \colon X\to X^t$, we let $\nu(\theta)$ be the positive integer such that $\deg(\theta) = \nu(\theta)^2$, and we define 
\[
\Lambda_\theta = \Lambda\left[\frac{1}{\nu(\theta)}\right] = \bbZ\left[\frac{1}{\nu(\theta)\cdot (2g+d+1)!}\right]\, .
\]

\subsection{}\label{subsec:ell}
Consider the homomorphism $(\id,\theta) \colon X\to X\times_S X^t$, and define a class $\ell \in \CH^1_{(0)}(X;\Lambda)$ by
\[
\ell = \tfrac{1}{2} \cdot (\id,\theta)^*\bigl(\wp\bigr)\, ,
\]
which lies in~$\CH^1_{(0)}$ because $[n]_{X\times_S X^t}^*(\cP) \cong \cP^{n^2}$ for all~$n$. For instance, if $\theta$ is the polarization given by a symmetric relatively ample line bundle~$L$ on~$X$ then we have $\ell = c_1(L)$. (This explains why we put in the factor~$\frac{1}{2}$.)

For $m\geq 0$ we have $\ell^m \in \CH^m_{(0)}(X;\Lambda)$. In particular, $\ell^g \in \CH^g_{(0)}(X;\Lambda) = \CH_{0,(0)}(X/S;\Lambda) = \Lambda\cdot \bigl[e(S)\bigr]$, and by the standard Riemann--Roch results for abelian varieties we find that
\begin{equation}\label{eq:lg}
\ell^g = \nu(\theta) \cdot g! \cdot \bigl[e(S)\bigr]\, .
\end{equation}
Let us also note that $\pi_*(\ell^m) = 0$ for $m\neq g$ by what was explained in Section~\ref{subsec:idealI}.

\subsection{}\label{subsec:Fexpl}
We will now work with coefficients in the ring~$\Lambda_\theta$. Define $\lambda \in \CH^{g-1}_{(0)}(X;\Lambda_\theta)$ by
\[
\lambda = \frac{\ell^{g-1}}{\nu(\theta) \cdot (g-1)!}\, .
\]
Alternatively, we claim that $\lambda$ is the unique class in $\CH^{g-1}_{(0)}(X;\Lambda_\theta)$ such that $\cF(\ell) = (-1)^{g-1} \cdot \theta_*(\lambda)$. To see this, we can follow the arguments of~\cite{Kunnemann}, Proposition~2.2. Namely, we have
\[
\theta^*\cF\bigl(\exp(\ell)\bigr) = \theta^*\pr_{2,*}\Bigl(\pr_1^*\bigl(\exp(\ell)\bigr) \cdot \exp(\wp) \Bigr) = \theta^*\pr_{2,*}\Bigl(\exp\bigl(\pr_1^*(\ell) + \wp \bigr)\Bigr)\, .
\]
Now use that the diagrams
\[
\begin{tikzcd}
X \times_S X \ar[d,"\pr_2"'] \ar[r,"\id \times \theta"]& X \times_S X^t \ar[d,"\pr_2"] \\
X \ar[r,"\theta"] & X^t 
\end{tikzcd}
\qquad
\begin{tikzcd}
X\times_S X \ar[r,"m"] \ar[d,"\pr_2"'] & X\ar[d,"\pi"] \\
X \ar[r,"\pi"] & S
\end{tikzcd}
\]
are Cartesian, and that $(\id \times \theta)^*\bigl(\wp\bigr) = m^*(\ell) - \pr_1^*(\ell) - \pr_2^*(\ell)$. We then find that
\[
\theta^*\cF\bigl(\exp(\ell)\bigr) = \pr_{2,*}\bigl(\exp(m^*(\ell) - \pr_2^*(\ell))\bigr) = \pr_{2,*}\bigl(\exp(m^*(\ell))\bigr) \cdot \exp(-\ell) = \pi^*\pi_*\bigl(\exp(\ell)\bigr) \cdot \exp(-\ell)\, .
\]
But it follows from what we have seen in Section~\ref{subsec:ell} that $\pi^*\pi_*\bigl(\exp(\ell)\bigr) = \nu(\theta) \cdot [X]$, and because $\theta$ has degree~$\nu(\theta)^2$ we get
\[
\cF\bigl(\exp(\ell)\bigr) = \frac{1}{\nu(\theta)} \cdot \theta_*\bigl(\exp(-\ell)\bigr)\, .
\]
Taking components in codimension~$g-1$ gives the claim.

\begin{theorem}\label{thm:sl2}
Define endomorphisms $e$, $h$ and~$f$ of~$\CH(X;\Lambda_\theta)$ by
\begin{alignat*}{2}
&e(x) = \ell \cdot x &&\text{(intersection product),}\\
&h(x) = (2i-s-g) \cdot x\qquad &&\text{if $x\in \CH^i_{(s)}(X;\Lambda_\theta)$,}\\
&f(x) = \lambda \star x &&\text{(Pontryagin product).}
\end{alignat*}
\begin{enumerate}
\item\label{sl2rep} The operators $e$, $h$ and~$f$ define a representation of the Lie algebra $\fsl_2$ on $\CH(X;\Lambda_\theta)$. Each $\CH^\bullet_{(s)}(X;\Lambda_\theta) = \oplus_i\; \CH^i_{(s)}(X;\Lambda_\theta)$ is a sub-representation. 

\item\label{PrimDec} Let $\St$ denotes the standard $2$-dimensional representation of~$\fsl_2$ over~$\Lambda_\theta$. For every~$s$, define
\[
M_j = \bigl\{x \in \CH^\bullet_{(s)}(X;\Lambda_\theta) \bigm| \text{$h(x) = -j \cdot x$ and $f(x) = 0$}\bigr\}\, .
\]
Then
\[
\CH^\bullet_{(s)}(X;\Lambda_\theta) \cong \bigl(\Sym^0(\St) \otimes_{\Lambda_\theta} M_0\bigr) \oplus \bigl(\Sym^1(\St) \otimes_{\Lambda_\theta} M_1\bigr) \oplus \cdots \oplus \bigl(\Sym^g(\St) \otimes_{\Lambda_\theta} M_g\bigr)
\]
as representations of~$\fsl_2$.
\end{enumerate}
\end{theorem}

\begin{proof}
\ref{sl2rep} We view $\CH(X;\Lambda_\theta)$ as a graded group by declaring the summand $\CH^i_{(s)}(X;\Lambda_\theta)$ to be of degree $2i-s-g$. (Motivically, one would refer to this number as the weight.) By construction, this is the grading that corresponds to the action of~$h$. As it is clear that $e$ and~$f$ are homogeneous of degrees~$+2$ and~$-2$, respectively, we have $[h,e] = 2e$ and $[h,f] = -2f$, and the only thing that remains to be shown is that $[e,f] = h$. For this we can follow the arguments of K\"unnemann in~\cite{Kunnemann} and~\cite{KunnemannMotives}.

We view $X^2 = X\times_S X$ as an abelian scheme over~$X$ via the first projection. For $n\in \bbZ$, let $\Gamma_{[n]} \in \CH_0(X^2/X;\Lambda_\theta)$ be the class of the graph of $[n] \colon X\to X$. Note that $\Gamma_{[m]} \star \Gamma_{[n]} = \Gamma_{[m+n]}$, and that $\Gamma_{[1]}$ is nothing but the class of the diagonal. By Corollary~\ref{cor:Istarnu+1} we have $(\Gamma_{[1]}-\Gamma_{[0]})^{\star (2g+d+1)} = 0$, so that we can define
\begin{equation}\label{eq:logGamma1}
\log \Gamma_{[1]} = \sum_{j=1}^{2g+d}\; (-1)^{j-1} \cdot \frac{(\Gamma_{[1]}-\Gamma_{[0]})^{\star j}}{j} = \sum_{n=0}^{2g+d}\; c_n \cdot \Gamma_{[n]}\quad\text{with}\quad
c_n = \sum_{j=\max\{1,n\}}^{2g+d} \frac{(-1)^{n-1} \binom{j}{n}}{j}\, .
\end{equation}
Then we define correspondences $\pi_i \in \CH_0(X^2/X;\Lambda_\theta)$, for $i=0,\ldots,2g$, by
\[
\pi_i = \frac{1}{(2g-i)!}\cdot \Bigl(\log \Gamma_{[1]}\Bigr)^{\star (2g-i)}\, .
\]
It follows from the last expression in~\eqref{eq:logGamma1} that each $\pi_i$ is a linear combination with coefficients in~$\Lambda_\theta$ of correspondences of the form~$\Gamma_{[n]}$, say
\[
\pi_i = \sum_{n\in \bbZ}\; a_{i,n} \cdot \Gamma_{[n]}
\]
with $a_{i,n} \in \Lambda_\theta$ nonzero for finitely many pairs $(i,n)$.

K\"unnemann shows (\cite{KunnemannMotives}, Section~3) that the $\pi_{i,*}$ are the projectors that give the Beauville decomposition when working with $\bbQ$-coefficients. This implies that for every $i, j \in \{0,\ldots,2g\}$ we have the relations
\[
\sum_{n\in \bbZ}\; a_{i,n} \cdot n^j = \begin{cases} 1 & \text{if $j=2g-i$}\\ 0 & \text{otherwise.} \end{cases}
\]
It is immediate from this that the $\pi_i$ are also the projectors that give the Beauville decomposition with coefficients in~$\Lambda_\theta$, i.e., if $\alpha \in \CH^r_{(s)}(X;\Lambda_\theta)$ then we have
\[
\pi_{i,*}(\alpha) = \begin{cases} \alpha & \text{if $i=r-2s$}\\ 0 & \text{otherwise.} \end{cases}
\]
(For this step everything works with coefficients in~$\Lambda$.) The rest of the proof of the identity $[e,f]=h$ now works exactly the same as in~\cite{Kunnemann}. The last assertion of~\ref{sl2rep} is clear from the definitions. 

\ref{PrimDec} This follows from Proposition~\ref{prop:sl2rep} applied with $V = \CH_{(s)}(X;\Lambda_\theta)$, where we take $V_i = \oplus_{2k-s=g+i}\; \CH^k_{(s)}(X;\Lambda_\theta)$, which is the subspace of $\CH_{(s)}(X;\Lambda_\theta)$ on which $h$ acts as multiplication by~$i$.
\end{proof}

\begin{remark}
In part~\ref{PrimDec} of the theorem, we in fact have $M_j= 0$ if $j\not\equiv g+s\bmod{2}$, since $h$ is multiplication by $2i-s-g$ on~$\CH^i_{(s)}(X;\Lambda_\theta)$.
\end{remark}

\section{An application}\label{sec:applic}

In this section, we give an application of our mains results to torsion classes. We conder a $g$-dimensional abelian variety~$X$ over an algebraically closed field~$L$. (Note that in this case either~\ref{charL=0} or~\ref{lift} in~\ref{subsec:charL} is saitisfied.) As before, we assume $X$ has a polarization~$\theta$ of degree~$\nu(\theta)^2$, and we let $\Lambda_\theta = \bbZ\bigl[\frac{1}{(2g+1)!\, \nu(\theta)}\bigr]$.

\begin{lemma}
With assumption as above, we have an isomorphism of $\Lambda_\theta$-modules
\[
\CH^1_{(1)}(X;\Lambda_\theta) \cong X(L) \otimes \Lambda_\theta\, .
\]
\end{lemma}

\begin{proof}
We have $\CH^1(X) \cong \Pic(X)$. As $2 \in \Lambda_\theta^*$, every element $\alpha \in \Pic(X) \otimes \Lambda_\theta$ can be uniquely written as the sum
\[
\alpha = \frac{\alpha + [-1]^*(\alpha)}{2} + \frac{\alpha - [-1]^*(\alpha)}{2}
\]
of a symmetric element and an anti-symmetric element. This is precisly the Beauville decomposition $\CH^1(X,\Lambda_\theta) = \CH^1_{(0)}(X,\Lambda_\theta) \oplus \CH^1_{(1)}(X,\Lambda_\theta)$. As the anti-symmetric classes correspond to the elements in $\Pic^0(X) \otimes \Lambda_\theta = X^t(L) \otimes \Lambda_\theta \cong X(L) \otimes \Lambda_\theta$, this gives the assertion.
\end{proof}

\begin{corollary}
Let $i \in \{1,\ldots,g\}$. Then $\CH^i_{(1)}(X;\Lambda_\theta)$ contains a submodule that is isomorphic to $X(L) \otimes \Lambda_\theta$. In particular, if $p>2g+1$ is a prime number with $p\neq \mathrm{char}(L)$ and $p\nmid \nu(\theta)$ then $\CH^i_{(1)}(X;\Lambda_\theta)$ contains a subgroup that is isomorphic to $(\bbQ_p/\bbZ_p)^{2g}$.
\end{corollary}

\begin{proof}
For the first assertion, note that $\CH^1_{(1)}(X;\Lambda_\theta)$ is the space~$M_{1-g}$ of Theorem~\ref{thm:sl2}\ref{PrimDec} (taking $s=1$), and then it follows from that result that~$e^{i-1}$ induces an isomorphism of $\CH^1_{(1)}(X;\Lambda_\theta)$ to a subspace of~$\CH^i_{(1)}(X;\Lambda_\theta)$, for $i=1,\ldots,g$. (Cf.\ Proposition~\ref{prop:sl2rep}\ref{sl2rep1}.) The last assertion then follows by looking at the $p$-primary torsion subgroups. (Recall that $L=\overline{L}$.)
\end{proof}

It is nice to contrast this last result to what we know about Chow groups with rational coefficients; for instance, if $L$ is the algebraic closure of a finite field then $\CH^i_{(1)}(X;\bbQ)$ is zero for all~$i$.

\bigskip

\noindent
\texttt{junaid01@uw.edu}

\noindent
Department of Mathematics, University of Washington, Seattle, WA 98195, USA
\medskip

\noindent
\texttt{hazem.hassan@mail.mcgill.ca}

\noindent
Department of Mathematics and Statistics, McGill University, Montreal, Canada
\medskip

\noindent
\texttt{clin130@jhu.edu}

\noindent
Department of Mathematics, Johns Hopkins University, Baltimore, MD 21218, USA
\medskip

\noindent
\texttt{maniv013@umn.edu}

\noindent
Department of Mathematics, University of Minnesota, Twin Cities, Minneapolis, MN
55455, USA
\medskip

\noindent
\texttt{lily.mcbeath.gr@dartmouth.edu}

\noindent
Department of Mathematics, Dartmouth College, Hanover, NH 03755 USA
\medskip

\noindent
\texttt{b.moonen@science.ru.nl}

\noindent
Radboud University Nijmegen, IMAPP, Nijmegen, The Netherlands


\end{document}